\documentclass[12pt,a4paper]{amsart}
	
	\usepackage{amssymb,amsmath,amsthm,amsfonts,amscd,latexsym, dsfont, color, xcolor}
	\usepackage[hidelinks]{hyperref} 
	\usepackage{graphics}
	\usepackage{wrapfig}
	\usepackage{pst-node}
	\usepackage{tikz-cd} 
	\usepackage{enumitem}
	\usepackage{mathrsfs}
        \usepackage{svg}
	\newcommand{\bbR}{\mathbb R}
	\newcommand{\RP}{\mathbb R \mathbb P}

	\usepackage{cleveref}

\DeclareFontFamily{U}{mathx}{\hyphenchar\font45}
\DeclareFontShape{U}{mathx}{m}{n}{<-> mathx10}{}
\DeclareSymbolFont{mathx}{U}{mathx}{m}{n}
\DeclareMathAccent{\widebar}{0}{mathx}{"73}

\crefformat{section}{\S#2#1#3} 
\crefformat{subsection}{\S#2#1#3}
\crefformat{subsubsection}{\S#2#1#3}
	\allowdisplaybreaks

\theoremstyle{definition}

\usepackage{thmtools, thm-restate}
\theoremstyle{plain}

\newtheorem{theorem}{Theorem}[section]
\newtheorem{proposition}[theorem]{Proposition}
\newtheorem{lemma}[theorem]{Lemma}
\newtheorem{corollary}[theorem]{Corollary}
\newtheorem{question}[theorem]{Question}
\newtheorem{example}[theorem]{Example}

\newtheorem{definition}[theorem]{Definition}
\newtheorem*{remark}{Remark}
\newtheorem*{notation}{Notation}

\numberwithin{equation}{section}

	\usepackage[margin=1in]{geometry}
\begin{document}

\title{Leaves of Foliated Projective Strutures}
\author{Alexander Nolte}
\address{Rice University, Houston TX, USA}
\email{alex.nolte@rice.edu}
\maketitle

\begin{abstract}
The $\text{PSL}(4,\mathbb{R})$ Hitchin component of a closed surface group $\pi_1(S)$ consists of holonomies of properly convex foliated projective structures on the unit tangent bundle of $S$. We prove that the leaves of the codimension-$1$ foliation of any such projective structure are all projectively equivalent if and only if its holonomy is Fuchsian. This implies constraints on the symmetries and shapes of these leaves.

We also give an application to the topology of the non-${\rm T}_0$ space $\mathfrak{C}(\RP^n)$ of projective classes of properly convex domains in $\RP^n$. Namely, Benz\'ecri asked in 1960 if every closed subset of $\mathfrak{C}(\RP^n)$ that contains no proper nonempty closed subset is a point. Our results imply a negative resolution for $n \geq 2$. 
\end{abstract}

\section{Introduction}

A $\text{PSL}(4,\mathbb{R})$ Hitchin representation $\rho$ of a closed surface group $\Gamma$ induces a curious $\Gamma$-invariant curve $\mathfrak{s}_\rho$ from the Gromov boundary $\partial \Gamma$ to the space $\mathfrak{C}$ of projective classes of properly convex domains in $\RP^2$. We call $\mathfrak{s}_\rho$ the \textit{leaf map} of $\rho$, and study it here.

As for other equivariant maps from $\partial \Gamma$ arising from geometry (e.g. \cite{benoist2004convexesI} \cite{bowen1979hausdorff} \cite{brigeman2022hessian} \cite{cannon2007group} \cite{guichard2005regularity} \cite{potrie2017eigenvalues-entropy}), the regularity and irregularities of $\mathfrak{s}_\rho$ are salient and interesting. The relevant aspects of our setting have an idiosyncratic character due to the point-set topological richness of $\mathfrak{C}$. Namely, $\mathfrak{C}$ is non-separated (i.e. not ${\rm T}_0$) and contains both large families of closed one-point sets and dense one-point sets \cite{benzecri1960thesis} \cite{goldman1990convex} \cite{kacVinberg1967quasihomogeneous}.

We prove $\mathfrak{s}_\rho$ is constant if and only if $\rho$ is Fuchsian. A proposition of Benoist \cite{benoist2023chat} then implies that, for non-Fuchsian $\rho$, images of leaf maps are closed in $\mathfrak{C}$, are not points, and are \textit{minimal} in the sense that they contain no proper nonempty closed subset. It follows that non-point minimal closed sets exist in the space $\mathfrak{C}(\RP^n)$ of projective classes of properly convex domains in $\RP^n$ ($n \geq 2$). The existence of non-point minimal closed sets is a basic question for a non-separated space. It has been open for $\mathfrak{C}(\RP^n)$ since Benz\'ecri posed the question in 1960 (\cite{benzecri1960thesis} \S V.3).

Let us be more detailed. By work of Guichard-Wienhard \cite{guichard2008convex}, $\text{PSL}(4,\bbR)$ Hitchin representations are exactly the holonomies of properly convex foliated projective structures on the unit tangent bundle $T^1(S)$, which are a refinement of $(\text{PSL}(4,\mathbb{R}), \mathbb{RP}^3)$ structures on $T^1S$ (see \S \ref{convex-projective-structures-background}). By definition, the developing map of such a projective structure maps leaves of the semi-stable geodesic foliation $\overline{\mathcal{F}}$ of $T^1\widetilde{S}$ to properly convex domains in projective planes. The leaf space of $\overline{\mathcal{F}}$ is identified with $\partial \Gamma$, and $\mathfrak{s}_\rho(x)$ is defined for $x \in \partial \Gamma$ as $[\text{dev}_\rho x] \in \mathfrak{C}$. For $\rho$ Fuchsian, $\mathfrak{s}_\rho$ is constant with value the ellipse.

Leaf maps exhibit counter-intuitive phenomena. For instance, $\mathfrak{s}_\rho$ maps any nonempty open set $U \subset \partial \Gamma$ onto all of $\mathfrak{s}_\rho(\partial \Gamma)$. In general, determining when leaf maps are constant is made difficult by the non-separation of $\mathfrak{C}$. We resolve the matter:

\begin{theorem}\label{theorem-fuchsian-locus-characterization}
    Let {\rm{$\rho \in \text{Hit}_4(S)$}}. The following are equivalent:
        \begin{enumerate}
            \item $\rho$ is Fuchsian,
            \item The leaf map $\mathfrak{s}_\rho$ is constant,
            \item The leaf map $\mathfrak{s}_\rho$ has countable image,
            \item There exists a leaf $\mathfrak{s}_\rho(x)$ that is divisible, a closed point of $\mathfrak{C}$, or has non-discrete projective automorphism group.
        \end{enumerate}
\end{theorem}

Recall that a properly convex domain is \textit{divisible} if it admits a cocompact action by a discrete subgroup of $\text{SL}(3,\bbR)$. Condition (4) considerably limits the symmetries of leaves of non-Fuchsian $\rho$. It is in contrast to the observation that some leaves have symmetries: automorphism groups of leaves $\mathfrak{s}_\rho(\gamma^\pm)$ of fixed points $\gamma^\pm \in \partial \Gamma$ for $\gamma \in \Gamma - \{e\}$ contain $\mathbb{Z}$.

Though non-constancy of $\mathfrak{s}_\rho$ for non-Fuchsian $\rho$ may appear intuitive, it implies leaf maps exhibit a rather dramatic phenomenon, impossible for any map to a ${\rm T}_1$ space: \begin{theorem}\label{headline-theorem-wild-pathology}
    For non-Fuchsian {\rm{$\rho \in \text{Hit}_4(S)$}}, the leaf map $\mathfrak{s}_\rho: \partial \pi_1 S \to \mathfrak{C}$ is continuous, constant on $\pi_1S$ orbits, and not constant.
\end{theorem}

Note in the above theorem that all $\pi_1S$ orbits in $\partial \pi_1 S$ are dense.

Benoist has proved that $\mathfrak{s}_\rho$ has closed image in $\mathfrak{C}$ (in unpublished work; see \S \ref{section-thanks-Yves} for details). From the continuity of $\mathfrak{s}_\rho$ and the minimality of the action of $\Gamma$ on $\partial \Gamma$, it follows that the image $\mathfrak{s}_\rho(\partial \Gamma)$ is a \textit{minimal} closed set in $\mathfrak{C}$, in the sense that it is closed and contains no proper nonempty closed subset. By taking cones over leaves of non-Fuchsian $\text{PSL}(4,\bbR)$ Hitchin representations, non-point minimal closed subsets of $\mathfrak{C}(\RP^n)$ can be constructed for all $n \geq 2$ (\S \ref{subsubsection-putting-it-together}).

All prior examples of minimal closed sets in $\mathfrak{C}(\RP^n)$, such as divisible domains \cite{benzecri1960thesis}, are points. So our results imply:

\begin{theorem}\label{thm-solves-benzecri}
    For all $n \geq 2$, $\mathfrak{C}(\RP^n)$ contains minimal closed sets that are not points.
\end{theorem}

Benz\'ecri concludes his seminal thesis, in which his namesake compactness theorem is proved and the topology of $\mathfrak{C}$ is first seriously studied, with a few questions on $\mathfrak{C}(\RP^n)$ for $n \geq 2$ (\cite{benzecri1960thesis} \S V.3). The first was whether all minimal closed subsets of $\mathfrak{C}(\RP^n)$ are points.

Among the experts aware of $\mathfrak{s}_\rho$ having closed image, Theorems \ref{theorem-fuchsian-locus-characterization}-\ref{thm-solves-benzecri} were expected to be true. However, no proof that $\mathfrak{s}_\rho$ is non-constant for non-Fuchsian $\rho$ had been found. This ends up being the main difficulty, and presents technical challenges. Our proof uses a range of methods, for instance relying on the Baire category theorem, the classification of Zariski closures of Hitchin representations, and Benoist's limit cone theorem.

\subsubsection{Shapes of Leaves} Our results place further restrictions on the geometry of individual leaves of non-Fuchsian properly convex foliated projective structures, which we explain here. First, they prevent any boundary point of a leaf from being too regular without being very flat.

\begin{corollary}\label{cor-arbitrary-boundary-points}
    Let $\rho$ be non-Fuchsian and $x \in \partial \pi_1 S$. Then the leaf $\mathfrak{s}_\rho(x)$ has no $C^2$ boundary point of nonvanishing curvature. 
\end{corollary}

This is analogous to a classical result of Benz\'ecri for divisible domains. It is notable in that it constrains arbitrary boundary points. This is in contrast to the constraints accessible with standard methods to study boundary regularity of similar objects, which control the worst-behaved points (e.g. \cite{guichard2005regularity}  \cite{potrie2017eigenvalues-entropy} \cite{zhangZimmer2017regularity}).

Pairing Theorem \ref{theorem-fuchsian-locus-characterization} with the closedness of the collection of all leaves in $\mathfrak{C}$ results in constraints on how complicated and asymmetric the boundary behavior a leaf may be. For instance, Benz\'ecri showed in \cite{benzecri1960thesis} that there are dense one-point sets in $\mathfrak{C}$. The following implies that any such domain cannot occur as a leaf.

\begin{corollary}\label{cor-not-awful}
     If $\rho \in {\rm {Hit}}_4(S)$ is non-Fuchsian and $x \in \partial \Gamma$, then ${\rm{Cl}}_{\mathfrak{C}}\{\mathfrak{s}_\rho(x)\}$ contains no closed point.
\end{corollary}

In the remainder of the introduction we outline our proof and situate our results in the context of broader projects in higher Teichm\"uller theory.

\subsection{Outline of Proof of Thm. \ref{theorem-fuchsian-locus-characterization}.}  A rough outline of our proof is that after addressing regularity of varying projective equivalences with a Baire category argument, considerations of boundary regularity and convexity of leaves $\mathfrak{s}_\rho(x)$ constrain the eigenvalues of $\rho$ when $\mathfrak{s}_\rho(x)$ is constant. These constraints, when paired with deep work of Benoist on limit cones \cite{benoist1997proprietes} and the classification of Zariski closures of Hitchin representations \cite{sambarino2020infinitesimal} allow us to deduce Theorem \ref{theorem-fuchsian-locus-characterization}.

It proves useful to do case analysis on the size of the projective automorphism group of $\mathfrak{s}_\rho(x)$. The most involved case is when $\mathfrak{s}_\rho(x)$ has discrete automorphism group. This case is ill-suited to productive use of Benz\'ecri's compactness theorem, and is a place where we must contend with the non-separation of $\mathfrak{C}$. This appears in the form that there are discontinuous paths $A_t : [0,1] \to \text{SL}(3,\bbR)$ and domains $\Omega$ in $\mathbb{RP}^2$ so that $A_t \overline{ \Omega}$ is continuous in the Hausdorff topology.\footnote{The easy-to-deal-with example of this is to use projective automorphisms of $\Omega$. We must also contend with e.g. the possibility that for a divergent sequence $A_t \in \text{SL}(3,\bbR)$ the domains $A_t \overline{\Omega}$ converge to $\overline{\Omega}$.}

Our argument in this case to obtain constraints on eigenvalues of $\rho$ if $\mathfrak{s}_\rho$ is constant has two main parts. The first hinges on the Baire category theorem, and shows that the above pathology may be avoided on a nonempty open subset $U \subset \partial \Gamma$ in the sense that we may arrange for representatives of the equivalence classes $\mathfrak{s}_\rho(x)$ to vary by a continuous family of projective equivalences on $U$. This facilitates a ``sliding'' argument that places constraints on the boundary regularity of leaves $\mathfrak{s}_\rho(x)$, through comparison with the dynamics of the action of $\rho(\gamma)$ on the boundary of the domain.

The restrictions we obtain are equivalent to the logarithms of the eigenvalues of $\rho(\gamma)$ $(\gamma \in \Gamma)$ satisfying an explicit $\gamma$-independent homogeneous polynomial. The endgame of our proof is to show that the only way this constraint may be satisfied is if $\rho$ is $4$-Fuchsian. We use two substantial results here, namely Guichard's classification of Zariski closures of Hitchin representations (see \cite{sambarino2020infinitesimal}) and a deep theorem of Benoist \cite{benoist1997proprietes} on limit cones of Zariski-dense representations.

\subsection{Context and Related Results}

\subsubsection{Properly Convex Projective Structures}\label{history-projective-structures}

Some notable analogues to Theorems \ref{theorem-fuchsian-locus-characterization} and \ref{headline-theorem-wild-pathology} occur in the study of properly convex projective structures on surfaces. These structures parameterize $\text{SL}(3,\mathbb{R})$ Hitchin components (\cite{goldman1990convex}, \cite{choi1993convex}).

Briefly, a projective structure $(\text{dev}, \text{hol})$ on $S$ is said to be \textit{properly convex} if $\text{dev}$ is a homeomorphism of $\widetilde{S}$ onto a properly convex domain $\Omega$ of $\mathbb{RP}^2$. In this case, $\Gamma$ acts properly discontinuously and without fixed-points on $\Omega$ through $\text{hol}$.

A similar statement to Theorem \ref{headline-theorem-wild-pathology} that is much easier to prove is the observation that in the above notation, $\partial \Omega$ is topologically a circle and the map $\text{reg}: \partial \Omega \to (1,2]$ associating to $x \in \partial \Omega$ the regularity of $\partial \Omega$ at $x$ (see e.g. \S\ref{background-section-convex-domains}) is a $\Gamma$-invariant map that is constant on all orbits of $\Gamma$, and only constant if $\text{hol}$ is in the Fuchsian locus of $\text{Hit}_3(S)$.

Of course this is an imperfect analogue to Theorem \ref{headline-theorem-wild-pathology} since the target, $(1,2]$, of $\text{reg}$ is much better-separated than $\mathfrak{C}$, and there is no aspect of continuity present. Nevertheless, there is a theme here that the local projective geometry of domains of discontinuity for non-Fuchsian $\text{PSL}(n,\mathbb{R})$ Hitchin representations is quite complicated (c.f. also \cite{pozzettiSambarino2022lipschitz}).

The geometry of properly convex projective structures is well-studied, and much of the structure in this setting (e.g. \cite{benoist2004convexesI} \cite{guichard2005regularity}) is due to the presence of divisibility. It is not clear to what extent the geometry of leaves $\mathfrak{s}_\rho(x)$ is similar. One expects similarities due to the closedness of the image of $\mathfrak{s}_\rho$.

\subsubsection{Geometric Structures and Hitchin Representations}

For all split real forms $G$ of complex simple centerless Lie groups, the $G$-Hitchin components are parametrized by holonomies of connected components of spaces of geometric structures on manifolds $M_G$ associated to $S$ \cite{guichard2012anosov}. Understanding the qualitative geometry of these geometric structures is a program within higher rank Teichm\"uller theory, into which this work falls. The basic question of the topological type of $M_G$ has seen major recent progress in cases of special interest in \cite{alessandriniDavaloLi2021projective} and more generally in \cite{alessandrini2023fiber} and \cite{davalo2023nearly}. There is no qualitative characterization of these connected components of geometric structures currently known in general.

In fact, the only Lie group $G$ as above of rank at least $3$ where $M_G$ is known and the geometric structures corresponding to Hitchin representations are qualitatively characterized is $\text{PSL}(4, \mathbb{R})$. Since the analytic tools that are often used to study these geometric structures in low rank (e.g. \cite{collierToulisseTholozan2019geometry}) break down in rank $3$ \cite{sagman2022unstable}, the $\text{PSL}(4, \mathbb{R})$ Hitchin component is a natural candidate for study in developing expectations for the general geometry of Hitchin representations. 

\subsubsection{The Mapping Class Group Action on Hitchin Components}

A long-standing question in higher Teichm\"uller theory is to understand the structure of the action of the mapping class group $\text{Mod}(S)$ on Hitchin components. A conjecture that would have settled this question was due to Labourie \cite{labourie2008cross}. Labourie's conjecture holds for Hitchin components for Lie groups $G$ as above of rank $2$ \cite{labourie2017cyclic}, and was disproved in rank at least $3$ as the culmination of a series of papers by Markovi\'c, Sagman, and Smillie \cite{markovic2022non} \cite{markovic2022unstable} \cite{sagman2022unstable}.

However, the negative resolution to Labourie's conjecture does not appear to directly yield information about the $\text{Mod}(S)$ action on Hitchin components, and leaves open what we shall call the fibration conjecture (\cite{wienhard2018Invitation}, Conjecture 14). To state the fibration conjecture, let $\mathcal{Q}^k(S)$ denote the holomorphic bundle over Teichm\"uller space of holomorphic $k$-adic differentials (see e.g. \cite{bers1961holomorphic}).

\begin{question}[Fibration Conjecture]
    Is the {\rm{$\text{PSL}(n,\mathbb{R})$}} Hitchin component naturally {\rm{$\text{Mod}(S)$}}-equivariantly diffeomorphic to the bundle sum $\bigoplus_{k=3}^n \mathcal{Q}^k(S)$?
\end{question}

Work of the author \cite{nolte2022canonical} implies that a conjecture of Fock and Thomas on higher degree complex structures \cite{fock2021higher} is equivalent to the fibration conjecture. The connection of the fibration conjecture to this paper is through its prediction that there should be canonical projections $\text{Hit}_n(S) \to \text{Hit}_k(S)$ for $2 \leq k < n$. The only known such projections have $k=2$ (e.g. \cite{labourie2007flat}, \cite{loftin2001affine}, \cite{sapirHensel2021projection}).

In their paper \cite{guichard2008convex} introducing properly convex foliated projective structures, Guichard and Wienhard suggest that perhaps these geometric objects could be used to approach the fibration conjecture for $\text{PSL}(4,\mathbb{R})$. The question that motivated the investigations leading to this paper was if examining the leaves of properly convex foliated projective structures gave rise to a projection $\text{Hit}_4(S) \to \text{Hit}_3(S)$. This would have been evidence in favor of the Fock-Thomas and fibration conjectures.

More specifically, properly convex subsets of $\mathbb{RP}^2$ are the setting of the geometric structures corresponding to the $\text{SL}(3, \mathbb{R})$ Hitchin component, and also appear as leaves of properly convex foliated projective structures. One might hope, after noticing that $\mathfrak{s}_\rho$ is continuous and constant on $\Gamma$-orbits that $\mathfrak{s}_\rho$ was constant, $\mathfrak{s}_\rho(x)$ was divisible, and examining the action of $\rho \in \text{Hit}_4(S)$ on the value of $\mathfrak{s}_\rho(x)$ gave an element of $\text{Hit}_3(S)$. Theorem \ref{theorem-fuchsian-locus-characterization} shows that this hope fails.
  \vspace{0.25cm}

\par \noindent \textbf{Organization.} Following the introduction are two sections on background: \S \ref{background-section-convex-domains} on convex domains in $\mathbb{RP}^2$ and \S \ref{convex-projective-structures-background} on Hitchin representations and properly convex foliated projective structures. In \S \ref{section-proof-of-main-theorem} we prove Theorems \ref{theorem-fuchsian-locus-characterization}-\ref{thm-solves-benzecri} and present a proof, following Benoist and printed here with his permission, that $\mathfrak{s}_\rho(\partial \Gamma)$ is closed in $\mathfrak{C}$.

\vspace{0.25cm}

\par \noindent \textbf{Acknowledgements.} This paper would not have been written were it not for the reading group on Anosov representations at Rice University in $2021$ and Rice's RTG geometry-topology seminar. Among the participants of these, I would like to specifically thank Chris Leininger, Mike Wolf, Alan Reid, and Sara Edelman-Mu\~noz.

This paper has benefitted a great deal from conversations with various mathematicians, in particular with Olivier Guichard, Max Riestenberg, Jean-Philippe Burelle, Colin Davalo, and Teddy Weisman. I would like to thank Yves Benoist for sharing his proof that leaf maps have closed image, and his permission to print it here. It is my pleasure to further thank Mike Wolf for his support and guidance.

This material is based upon work supported by the National Science Foundation under Grant No. 1842494 and Grant No. 2005551.

\section{Properly Convex Domains in $\mathbb{RP}^2$}\label{background-section-convex-domains}

In this section we recall the foundational facts about properly convex subsets of $\mathbb{RP}^2$ that are essential to our later arguments. In particular, \S\ref{subsection-spaces-of-convex-sets} discusses spaces of properly convex domains and Benz\'ecri's compactness theorem, and \S\ref{benoist-lemma-statement} concerns a boundary regularity and convexity fact due to Benoist.

We begin by introducing definitions and notation. A set $\Omega \subset \mathbb{RP}^2$ is \textit{convex} if for any pair of points $p,q \in \Omega$ there is a line segment contained in $\Omega$ between $p$ and $q$. A \textit{domain} is an open connected subset of $\mathbb{RP}^2$. A convex domain $\Omega$ is said to be \textit{properly convex} if $\overline{\Omega}$ is contained in a single affine chart, and is said to be \textit{strictly convex} if for every $p,q \in \overline{\Omega}$, a line segment connecting $p$ and $q$ in $\overline{\Omega}$ can be taken to be contained in $\Omega$ except at its endpoints.

\subsection{Spaces of Properly Convex Sets}\label{subsection-spaces-of-convex-sets}

Let $\mathcal{C}$ denote the collection of properly convex domains in $\mathbb{RP}^2$. Let $\mathcal{C}^*$ denote the collection of pointed properly convex domains in $\mathbb{RP}^2$, that is, pairs $(\Omega, p)$ where $\Omega \in \mathcal{C}$ and $p \in \Omega$. We give $\mathcal{C}$ the topology induced by the Hausdorff topology on closures, and $\mathcal{C}^*$ the topology induced from the product $\mathcal{C} \times \mathbb{RP}^2$.  We denote the quotients of $\mathcal{C}$ and $\mathcal{C}^*$ by the action of $\text{SL}(3, \mathbb{R})$ by $\mathfrak{C}$ and $\mathfrak{C}^*$, respectively.

The topology of $\mathfrak{C}$ only separates some points---one-point sets in $\mathfrak{C}$ need not be closed. This phenomenon plays a prominent role in this paper. A first example of non-closed points in $\mathfrak{C}$ is as follows.

\begin{example}
    Let $e_1, e_2, e_3$ be a basis for $\mathbb{R}^3$. Work in an affine chart containing $[e_1], [e_2],$ and $[e_3]$. Let $\Omega$ be a strictly convex domain contained in this affine chart preserved by {\rm{$A = \text{diag}(e^{\lambda}, e^{\eta}, e^{-\lambda - \eta})$}} for some $\lambda > \eta \geq 0$. For instance $\Omega$ may be an ellipse if $\eta = 0$.
    
    Let $\ell$ denote the line segment from $[e_1]$ to $[e_3]$ in this affine chart and $p \in \ell$. Let $\ell'$ denote the line determined by $[e_2]$ and $p$. Then $\ell'$ bisects $\Omega$. Let $\Omega'$ be the component of $\Omega$ containing $[e_3]$. Then $\Omega'$ is not projectively equivalent to $\Omega$ as its boundary contains a line segment, but $A^n \overline{\Omega'}$ converges to $\overline{\Omega}$ in the Hausdorff topology. So $[\Omega] \in \overline{ \{[\Omega'] \}}$.
\end{example}

The closures of points in $\mathfrak{C}$ vary a great deal: it is a consequence of Benz\'ecri's compactness theorem below that all divisible domains are closed points, while Benz\'ecri also showed (\cite{benzecri1960thesis} \S V.3, p.321) there there exist dense one-point sets in $\mathfrak{C}$. The topology of $\mathfrak{C}$ is quite complicated, and is rich enough that the continuity of a map with target $\mathfrak{C}$ has nontrivial content.

On the other hand, all of the poor separation in $\mathfrak{C}$ is caused by divergent sequences of elements of $\text{SL}(3, \mathbb{R})$ for the tautological reason that if $K \subset \text{SL}(3, \mathbb{R})$ is compact and $\Omega \in \mathcal{C}$, then the orbit of $\Omega$ under $K$ represents a single point in $\mathfrak{C}$. As a consequence, if one is able to gain finer control on a sequence $\Omega_n \in \mathcal{C}$ than convergence in $\mathcal{C}$, it can be tractable to understand the limiting projective geometry of $\Omega_n$ in spite of the non-separation of points in $\mathfrak{C}$.

The typical way this is done in practice, though which is not so useful in the following, is by gaining control over a single point of the domains $\Omega_n$ in question, working with the space $\mathfrak{C}^*$ instead of $\mathfrak{C}$. It follows from the below fundamental result of Benz\'ecri that this is enough to guarantee uniqueness of limits.

\begin{theorem}[Benz\'ecri Compactness]
 $\text{{\rm {SL}}}(3, \mathbb{R})$ acts properly and co-compactly on $\mathcal{C}^*$.
\end{theorem}

As an immediate corollary, we have:

\begin{corollary} $\mathfrak{C}^*$ is a compact Hausdorff space.
\end{corollary}

\subsection{Regularity and Convexity of Domains}\label{benoist-lemma-statement}

In this subsection, we describe the notion of boundary behavior of convex sets best adapted to our uses and a relevant circumstance in which this quantity may be computed explicitly at a boundary point.

\begin{definition} Let $C$ be a closed embedded $C^1$ curve in $\mathbb{R}^2$. For $1 < \alpha < \infty$, we say that $p \in C$ is a {\rm{$\alpha$-modelled point}} of $C$ if there is an open neighborhood $U$ in $C$ of $p$ and constant $C_U > 0$ so that for all $y \in U$, $(1/C_U) d(y,x)^\alpha \leq d(y, T_xC) \leq C_U d(y,x)^\alpha$.
\end{definition}

Here, the distance is the standard Euclidean distance. Note that $\alpha$-modelling of a point $p \in C$ is invariant under projective transformations, and the curve $y = |x|^\alpha$ is $\alpha$-modelled $(\alpha > 1)$. For $\alpha \in (1,2)$, a point being $\alpha$-modelled implies it is a $C^\alpha$ point and not $C^{\alpha'}$ for any $\alpha' > \alpha$. For $\alpha > 2$, a point being $\alpha$-modelled implies it is an $\alpha$-convex point (see e.g. \cite{benoist2004convexesI}, \cite{guichard2005regularity}) and not $\alpha'$-convex for any $\alpha' < \alpha$.

The following lemma is essentially contained in work of Benoist (\cite{benoist2004convexesI}, proof of Corollaire 5.3). The form we use here is slightly stronger and more general than the version stated there, and follows from a close examination of the argument given in \cite{benoist2004convexesI}. 

\begin{lemma}[Models at Fixed Points]\label{regularity-lemma} Let $\Omega \subset \mathbb{RP}^2$ be a properly convex, strictly convex domain preserved by {\rm{$A\in \text{GL}(3, \mathbb{R})$}} conjugate to {\rm$\text{diag}(\lambda_1, \lambda_2, \lambda_3)$} with $\lambda_1 > \lambda_2 > \lambda_3 > 0$. Write $l_i = \log \lambda_i $ for $i = 1,2,3$ and let $x_{A^+}$ denote the attracting fixed point of $A$ in $\mathbb{RP}^2$.

Then $x_{A^+} \in \partial \Omega$ is $\alpha$-modelled for $$\alpha = \frac{l_1 - l_3}{l_1 - l_2}.$$
\end{lemma} 

In consideration of the importance of this lemma to the present work and the standing assumption of divisibility in \cite{benoist2004convexesI} (which we will show is false in general in our setting) we present a proof of Lemma \ref{regularity-lemma} in the appendix.

\section{Properly Convex Foliated Projective Structures and Hitchin Representations}\label{convex-projective-structures-background}

In this section, we recall the relevant features of Hitchin representations and the theory of properly convex foliated projective structures developed by Guichard and Wienhard in \cite{guichard2008convex} to our later discussion. We also prove a few basic lemmata, and set conventions for later use. \S \ref{super-elementary-remarks} is the only portion of this section not contained in existing literature.

\begin{notation} Let $S$ be a closed, oriented surface of genus $g \geq 2$, $\Gamma = \pi_1(S),$ and $\overline{\Gamma} = \pi_1(T^1(S))$. Let $\mathcal{G}, \mathcal{F}$ denote the stable and semi-stable geodesic foliations of $T^1S$. Let $\overline{\mathcal{F}}, \overline{\mathcal{G}}$ denote the lifts of $\mathcal{F}, \mathcal{G}$ to $T^1 \widetilde{S}$, and $\widetilde{\mathcal{G}}, \widetilde{\mathcal{F}}$ the lifts of $\mathcal{F}, \mathcal{G}$ to the universal cover of $T^1S$.
\end{notation}

\subsection{Hitchin Representations and Hyperconvex Fr\'enet Curves}

Hitchin representations $\Gamma \to \text{PSL}(n,\mathbb{R})$ are characterized in terms of the geometry of special equivariant curves by work of Labourie and Guichard \cite{labourie2006anosov}, \cite{guichard2008composantes}. This perspective is central to our methods, and we recall it here.

For $1\leq k \leq n$, denote the $k$-Grassmannian of $\mathbb{R}^n$ by $\text{Gr}_k(\mathbb{R}^n)$. A continuous curve $\xi  = (\xi^1, ..., \xi^{n-1}): \partial \Gamma \to \bigoplus_{k=1}^{n-1} \text{Gr}_k(\mathbb{R}^n)$ is a \textit{hyperconvex Fren\'et curve} if:
\begin{enumerate}
    \item (Convexity) For any $k_1,...,k_j$ with $\sum_{l=1}^j k_l \leq n$, and distinct $x_1,..., x_j \in \partial \Gamma$, the vector space sum $\xi^{k_1}(x_1) + ... + \xi^{k_j}(x_j)$ is direct;
    \item (Osculation) For any $x \in \partial \Gamma$ and $k_1,...,k_j$ with $K = \sum_{l=1}^j k_l < n$ we have that $\xi^K(x) = \lim\limits_{m \to \infty} \left[ \xi^{k_1}(x_1^m) \oplus ... \oplus \xi^{k_j}(x_j^m) \right]$ for any sequence $(x_1^m,...,x_j^m)$ of $j$-ples of distinct points so that for all $l$, the sequence $x_l^m $ converges to $x$.
\end{enumerate}

A hyperconvex Fren\'et curve $(\xi^1, ..., \xi^{n-1})$ is entirely determined by $\xi^1$. The standard example of such a curve is the Veronese curve $\xi^1: \mathbb{RP}^1 \to \mathbb{RP}^n$, given in the model of $\mathbb{R}^k$ $(k = 2,n+1)$ as homogeneous polynomials on $\mathbb{R}^2$ of degree $k-1$ by $[f] \mapsto [f^n]$. The relevant result to us here, which serves as our working definition of a Hitchin representation, is:

\begin{theorem}[Labourie \cite{labourie2006anosov}, Guichard \cite{guichard2008composantes}] A representation {\rm{$\rho: \Gamma \to \text{PSL}(n,\bbR)$}} is Hitchin if and only if there exists a $\rho$-equivariant hyperconvex Fren\'et curve.
\end{theorem}

We denote the $\text{PSL}(n,\bbR)$ Hitchin component(s) by $\text{Hit}_n(S)$. A fact that will be useful to us is that Hitchin representations $\rho: \Gamma \to \text{PSL}(n,\bbR)$ may always be lifted to $\text{SL}(n,\bbR)$.

Though the definition of a hyperconvex Fren\'et curve is stated in terms of sums of $\xi^k$, work of Guichard \cite{guichard2005dualite} shows that intersections of $\xi^k$ are also quite well-behaved, which is often the way in which we interact with the Fren\'et property.

\begin{proposition}[Guichard \cite{guichard2005dualite}]\label{Intersection-lemmas}
    Let $\xi = (\xi^1, ..., \xi^{n-1})$ be a hyperconvex Fren\'et curve. Then:
    \begin{enumerate}
        \item (General Position) If $n = \sum_{i=1}^j k_i$ and $x_1, ..., x_j \in \partial \Gamma$ are distinct, then $$\bigcap_{i=1}^j \xi^{n-k_i}(x_i) = \{0\};$$
        \item (Dual Osculation) For any $x \in \partial \Gamma$ and $k_1,...,k_j$ with $K = \sum_{l=1}^j k_l < n$ we have that for any sequence $(x_1^m,...,x_j^m)$ of $j$-ples of distinct points in $\partial \Gamma$ so that $x_l^m $ converges to $x$ for each $l$, $$\xi^{n-K}(x) = \lim\limits_{m \to \infty} \bigcap_{i=1}^j \xi^{k_i}(x_i^m).$$
    \end{enumerate}
\end{proposition}

\subsection{Properly Convex Foliated Projective Structures}

In this subsection, we recall the properties of geodesic foliations on surfaces that make the definition of properly convex foliated projective structures on surfaces well-defined, state their definitions and basic properties, and collect the main results of Guichard and Wienhard in \cite{guichard2008convex}. Our notation and the content here follows \cite{guichard2008convex}.

\subsubsection{Geodesic Foliations}

Fixing a hyperbolic metric on $S$ identifies the geodesic foliations of $T^1 \widetilde{S}$ and $T^1\mathbb{H}^2$, and identifies $\partial \Gamma$ with $\partial \mathbb{H}^2$. There is a well-known description of $T^1\mathbb{H}^2$ as orientation-compatible triples $(t_+, t_0,t_-)$ of distinct points in $\partial \Gamma$. We denote the space of such triples $\partial \Gamma^{(3)+}$. One obtains this identification by associating to $(p,v) \in T^1(S)$ the endpoints at infinity of the geodesic $\ell$ determined by $v$ as $t_-, t_+$, and the endpoint $t_0$ of the geodesic perpendicular to $\ell$ at $p$ that makes $(t_+, t_0,t_-)$ orientation-compatible (see Figure \ref{Unit-tangent-bundle-picture}).

Under this identification, the leaves of the semi-stable geodesic foliation $\overline{\mathcal{F}}$ are the collections of elements of $\partial \Gamma^{(3)+}$ with fixed $t_+$ entry, and the leaves of the stable geodesic foliation $\overline{\mathcal{G}}$ are the collections of elements of $\partial \Gamma^{(3)+}$ with fixed $t_-$ and $t_+$ entries. So the leaf spaces of $\overline{\mathcal{F}}$ and $\overline{\mathcal{G}}$ are identified with $\partial \Gamma$ and $\partial \Gamma^{(2)} := \Gamma \times \Gamma - \{(x,x) \mid x \in \Gamma\}$. In the following, we shall identify elements of $\partial \Gamma$ and $\partial \Gamma^{(2)}$ and the corresponding leaves of $\overline{\mathcal{F}}, \overline{\mathcal{G}}$.

\begin{figure}
    \centering
    \includegraphics[scale=0.40]{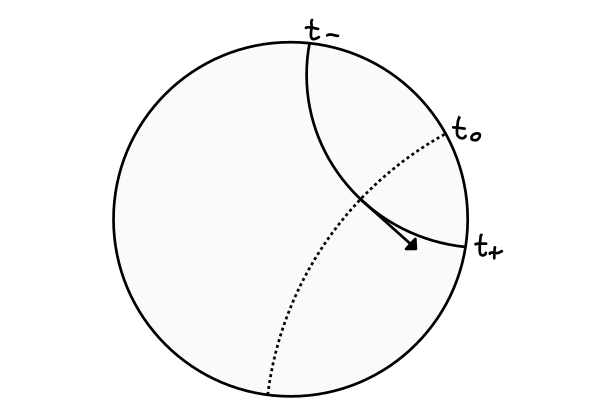} \vspace{-0.5cm}
    \caption{The unit tangent bundle $T^1\mathbb{H}$.}\label{Unit-tangent-bundle-picture}
\end{figure}

 This identification between $T^1 \widetilde{S}$ and $ \partial \Gamma^{(3)+}$ is equivariant with respect to the natural actions of $\Gamma$, and as a consequence, the topological type of the pair $(\mathcal{F}, \mathcal{G})$ is independent of the choice of hyperbolic metric.

\subsubsection{Properly Convex Foliated Projective Structures}

Consider $T^1S$, together with its stable and semi-stable foliations $(\mathcal{F}, \mathcal{G})$. Let $\mathcal{P}(S)$ denote the collection of projective structures on $T^1S$.

\begin{definition} Let $P$ be a projective structure on $T^1S$, viewed as an atlas of charts $\{ (U, \varphi_U)\}$ to $\mathbb{RP}^3$ with projective transisitons. Denote (a representative of) the developing data of $P$ as {\rm{$(\text{dev}, \text{hol})$}}.
\begin{enumerate} \item $P$ is {\rm{foliated}} if given any chart $(U, \varphi_U)$ and $v \in U$ contained in the leaves $g_v \cap U \in \mathcal{G}|_U$ and $f_v \cap U \in \mathcal{F}|_U$, then $\varphi_U(g_v \cap U)$ is contained in a projective line and $\varphi_U(f_v \cap U)$ is contained in a projective plane.
\item If for any leaf $f \in \widetilde{\mathcal{F}}$, the developed image {\rm{$\text{dev}(f)$}} is a properly convex domain in a projective plane, then we say $P$ is {\rm{properly convex}}.
\item Two foliated projective structures $P, P'$ are said to be equivalent if there is a homeomorphism $h$ of $T^1S$ isotopic to the identity that is a projective equivalence $h$ of $P$ and $P'$ so that $h^* \mathcal{F} = \mathcal{F}$ and $h^* \mathcal{G} = \mathcal{G}$.
\item Let $\mathcal{P}_f(S)$ and $\mathcal{P}_{pcf}(S)$ denote the collections of equivalence classes of foliated and properly convex foliated projective structures on $T^1S$, respectively.
\end{enumerate}
\end{definition}

Note that it is not clear that the natural mappings of $\mathcal{P}_{f}(S)$ and $\mathcal{P}_{pcf}(S)$ to $\mathcal{P}(S)$ given by forgetting the extra structure are injective, since the equivalence relation is refined. Developing maps of properly convex foliated projective structures always factor through $T^1\widetilde{S}$ as a consequence of \cite{guichard2008convex}, so we may work with $\overline{\mathcal{F}}$ and $\overline{\mathcal{G}}$ in place of $\widetilde{\mathcal{F}}$ and $\widetilde{\mathcal{G}}$.

Let $p: \overline{\Gamma} \to \Gamma$ be the map induced by the projection $T^1(S) \to S$. In \cite{guichard2008convex}, it is proved that for properly convex foliated projective structures $(\text{dev}, \text{hol})$, the value of $\text{hol}(\gamma)$ ($\gamma \in \overline{\Gamma})$ depends only on $p(\gamma)$. So any properly convex foliated projective structure induces a representation $\text{hol}_*: \Gamma \to \text{PSL}(4,\bbR)$, well-defined up to conjugacy. Write by $[\text{hol}_*]$ the associated conjugacy class of representations. In \cite{guichard2008convex} the following characterization of properly convex foliated projective structures in terms of the $\text{PSL}(4, \mathbb{R})$ Hitchin component is proved.

\begin{theorem}[Guichard-Wienhard \cite{guichard2008convex}]\label{classification-theorem}
    The holonomy map {\rm{$\mathcal{P}_{pcf}(S) \to \text{Hit}_4(S)$}} given by {\rm{$(\text{dev}, \text{hol}) \mapsto [\text{hol}_*]$}} is a homeomorphism.
\end{theorem}

The main definition to our investigations is:

\begin{definition}
    Given a properly convex foliated projective structure induced by a representation $\rho$, under the natural identification of the leaf space of $\overline{\mathcal{F}}$ and $\partial \Gamma$, we denote by $\mathfrak{s}_\rho(x)\in \mathfrak{C}$ the projective equivalence class of {\rm ${\text{dev}(x)}$ }. We call $\mathfrak{s}_\rho$ the {\rm{leaf map}} of $\rho$.
\end{definition}

One useful tool developed by Guichard and Wienhard in their proof that all Hitchin representations induce properly convex foliated projective structures is an explicit description of the developing map of the associated projective structure in terms of the hyperconvex Fren\'et curve $\xi = (\xi^1, \xi^2, \xi^3)$. See Figure \ref{picture-of-dev}, and discussion below.

To be more explicit, fix a Hitchin representation $\rho :\Gamma \to \text{PSL}(4,\bbR)$, and denote the corresponsing equivariant Fren\'et curve by $\xi = (\xi^1, \xi^2, \xi^3)$. Using the identification of $\partial \Gamma$ with the leaf space of $\overline{\mathcal{F}}$, denote semi-stable leaves of the geodesic foliation on $T^1 \widetilde{S}$ by $x \in \partial \Gamma$.

\begin{figure}
\begin{center} \makebox[0pt]{\, \, \, \, \, \, \, \, \, \, \, \, \, \,  \includegraphics[scale=0.16]{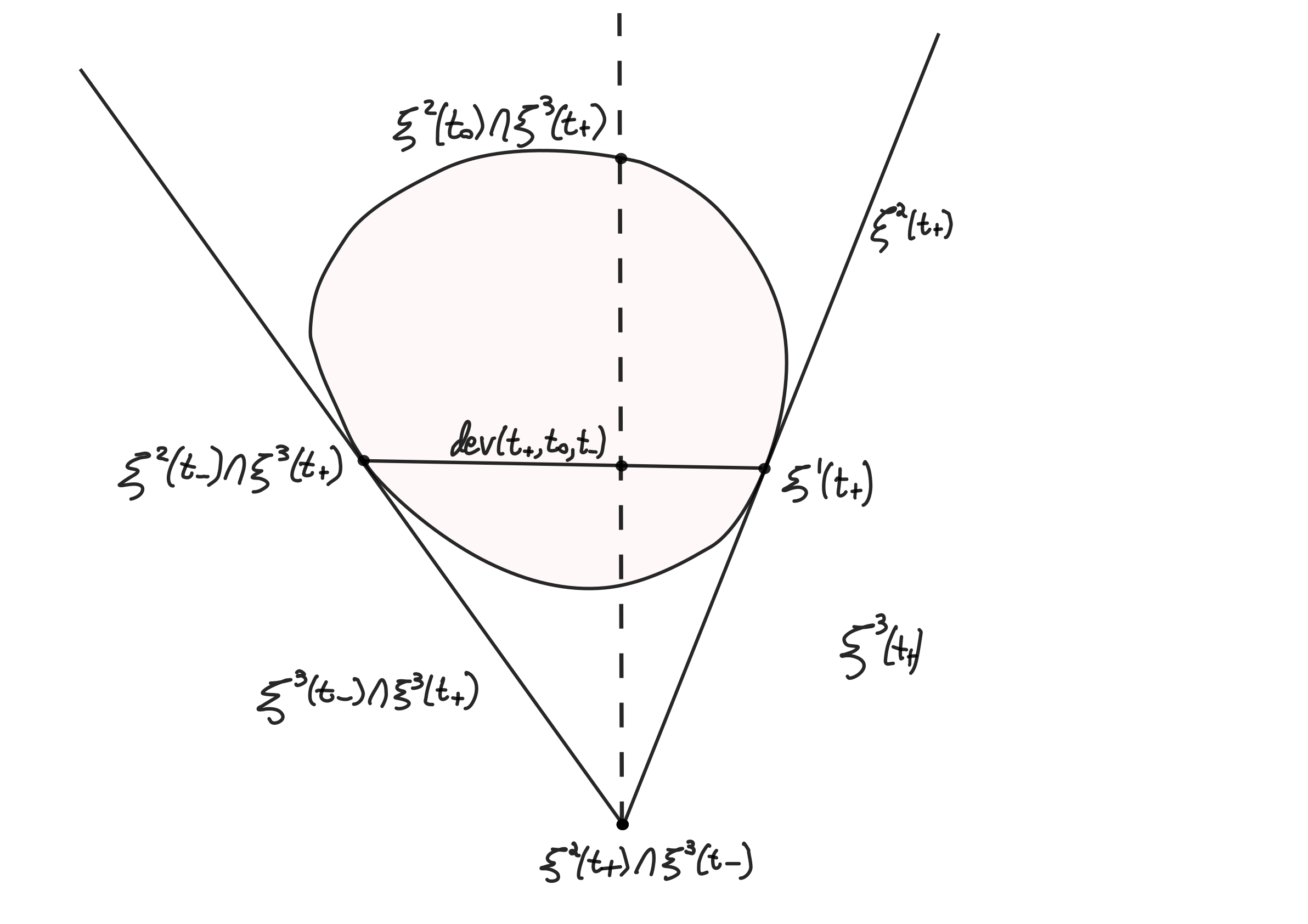}} \vspace{-0.5cm}
\end{center}
\caption{The developing map in terms of the Fren\'et curve.}\label{picture-of-dev}
\end{figure}

Following the notation of Guichard-Wienhard, define the two-argument map $\xi^1: \partial \Gamma \times \partial \Gamma \to \RP^3$ by 
\begin{align*} \xi^1_{t}(t') = \begin{cases} \xi^3(t) \cap \xi^2(t') & t \neq t' \\ 
\xi^1(t) & t = t'
\end{cases}. \end{align*} Then we can define the developing map of the projective structure we seek as
\begin{align*}
    \text{dev}:\qquad \partial \Gamma^{3+} &\to \RP^3 \\
    (t_+, t_0, t_-) &\mapsto \overline{\xi^1(t_+) \xi^1_{t_+}(t_-)} \cap \overline{\xi^1_{t_-}(t_{+}) \xi^1_{t_+}(t_0)},
\end{align*}where we denote the line in $\mathbb{RP}^3$ determined by two points $a$ and $b$ by $\overline{ab}$. Write $\Omega_\rho:= \text{dev}(\partial \Gamma^{(3)+})$.

A few qualitative remarks are in order. Here the boundary of $\Omega_\rho$ is given by $\partial \Omega = \bigsqcup_{t \in \partial \Gamma} \xi^2(t)$, where disjointness is a consequence of hyperconvexity. For any $x \in \partial \Gamma$, the leaf $\text{dev}(x)$ is $\xi^3(x) \cap \Omega_\rho$. The boundary of $\text{dev}(x)$ is given by $\{\xi^2(y) \cap \xi^3(x) : y \neq x \} \cup \{\xi^1(x)\}$. A supporting line to $\partial \text{dev}(x)$ at $\xi^2(y) \cap \xi^3(x)$ is $\xi^3(y) \cap \xi^3(x)$, and a supporting line to $\partial \text{dev}(x)$ at $\xi^1(x)$ is $\xi^2(x)$. These lines do not intersect $\text{dev}(x)$ due to the general position property of Fren\'et curves and our description of the boundary of $\text{dev}(x)$. We shall show in \S \ref{super-elementary-remarks} that these supporting lines are unique.

\begin{figure}[t] \vspace{-1cm}
\begin{center} \makebox[0pt]{\, \, \,\, \, \, \, \, \includegraphics[scale=0.28]{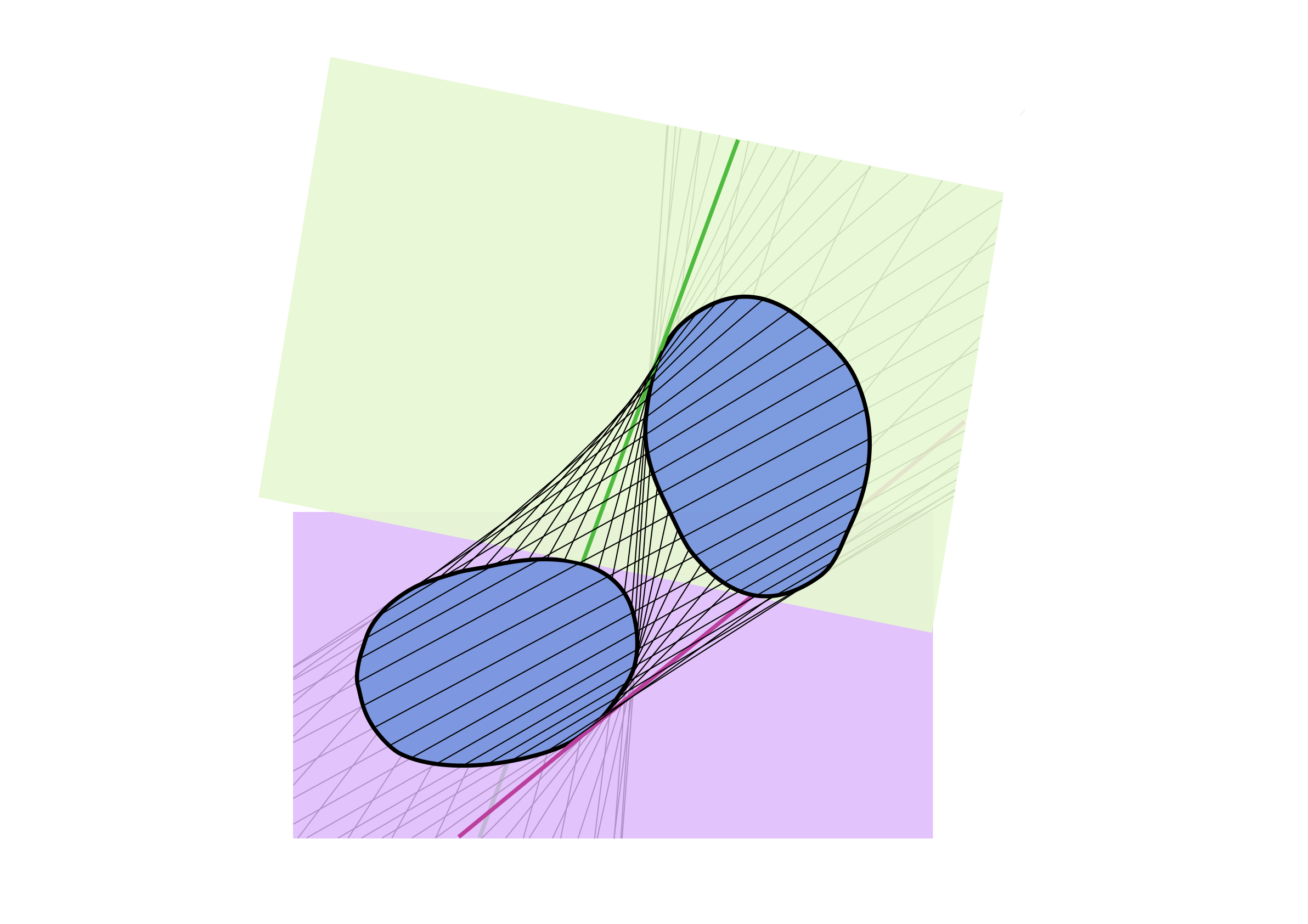}}
\end{center}\vspace{-1.5cm}
\caption{Sketch of relevant aspects of two slices of a domain of discontinuity and the ruling of the boundary by lines.}\label{figure-two-slices}
\end{figure}

\subsection{Two Remarks on Boundaries of Leaves}\label{super-elementary-remarks}

In this subsection, we describe two basic geometric features of the leaves $\text{dev}(x)$.

Our first observation is that the ruling of the boundary of $\Omega$ by $\xi^2(x)$ ($x \in \partial \Gamma$) gives rise to natural identifications of boundaries of leaves $\partial \text{dev}(x)$. Geometrically, any boundary point $p$ of $\text{dev}(x)$ is contained in exactly one $\xi^2(y)$ for $y \in \partial \Gamma$. Given another $x' \in \partial \Gamma$, the identification of boundaries maps $p$ to the unique intersection of $\xi^2(y)$ with $\partial \text{dev}(x')$ (see Figure \ref{figure-two-slices}). The below expresses this symbolically.

\begin{definition}
    For $x,x'\in \partial \Gamma$, define a the map {\rm $\Xi_{x \to x'}: \partial \text{dev}(x) \to \partial \text{dev}(x')$} by $\Xi_{x\to x'}(\xi^1_x(y)) = \xi_{x'}^1(y)$ for $y \in \partial \Gamma$. 
\end{definition}

As a consequence of continuity of $\xi^1_{x}(y)$, which follows from dual osculation in Proposition \ref{Intersection-lemmas}, the maps $\Xi_{x \to x'}(y)$ vary continuously in $x,x'$, and $y$.

Our second observation concerns the structure of the boundary of $\partial \text{dev}(x)$ for $x \in \partial \Gamma$: they are strictly convex and $C^1$. Strict convexity, in particular, is a tool that we use for some obstructions later.

\begin{proposition}[Basic Regularity]
     For all $x \in \partial \Gamma$, the leaf {\rm{$\text{dev}(x)$}} is strictly convex and has $C^1$ boundary.
\end{proposition}

\begin{proof}
To show that $\text{dev}(x)$ is $C^1$, we consider the dual properly convex domain $\text{dev}(x)^* \subset \xi^3(x)^*$. The boundary $\partial \text{dev}(x)^*$ is a topological circle consisting of supporting lines to $\partial \text{dev}(x)$. The path $\partial \Gamma \to \partial \text{dev}(x)^*$ given by $$y \mapsto \begin{cases} (\xi^3(y) \cap \xi^3(x))^* & y \neq x \\ \xi^2(x)^* & y = x 
\end{cases}$$ is a continuous injection of $\partial \Gamma \cong S^1$ into $\partial \text{dev}(x)^* \cong S^1$, and so must be surjective. So all supporting lines to $\text{dev}(x)$ must be of the form $\xi^3(y) \cap \xi^3(x)$ or $\xi^2(x)$. In particular, all boundary points of $\text{dev}(x)$ have unique tangent lines, which implies $\partial \text{dev}(x)$ is $C^1$.

Strict convexity follows from the general position property of Fren\'et curves as follows. Supposing otherwise, $\partial \text{dev}(x)$ must contain an interval $I$, contained in a line $\ell_I$. For any $y \neq x \in\partial \Gamma$ so that $\xi_x^1(y)$ is in the interior of $I$, we must have $\xi^3(y) \cap \xi^3(x) = \ell_I$, as this is a supporting line to $\partial \text{dev}(x)$ at a point in $I$. This is impossible by the general position property of Fren\'et curves, and proves strict convexity.
\end{proof}

\section{Proofs of the Main Theorems}\label{section-proof-of-main-theorem}

In this section we prove our main theorems. The vast majority of the effort is spent showing $\mathfrak{s}_\rho$ is not constant unless the Hitchin representation $\rho$ is $4$-Fuchsian. We begin by setting notation in \S \ref{notations-conventions-definitions}. An outline of the structure of the core of our proofs is then given in \S \ref{outline-of-main-proof}, and the remainder of the paper is spent following this outline.

\subsection{Notation, Conventions, and Definitions.}\label{notations-conventions-definitions}
Let us begin by setting up notation to facilitate comparison of projective types of leaves.

The group $\text{SL}(3,\bbR)$ acts simply transitively on $4$-ples of points in general position in $\mathbb{RP}^2$. So, by fixing a point $t_0 \in \partial \Gamma$ and a continuously varying family of $4$ points $$\{(p_1(t), p_2(t), p_3(t), p_4(t)) \mid  t \in \partial \Gamma \} \subset \mathbb{RP}^3$$ so that $p_i(t) \in \xi^3(t)$ ($i=1,...,4)$ and the points $(p_1(t), p_2(t), p_3(4), p_4(t))$ are in general position within $\xi^3(t)$ for all $t \in \partial \Gamma$, we induce well-determined projective equivalences $\xi^3(t) \to \xi^3(t_0)$ for all $t \in \partial \Gamma$.

One way to produce such a normalization is to take $4$ distinct points $x_1, ..., x_4 \in \partial\Gamma$ and let $p_i(t)$ ($i =1,2,3,4)$ be the unique point of intersection between $\xi^2(x_i)$ and $\partial \text{dev}(t)$. The continuity of the points $p_i(t)$ results in such a normalization being continuous in the sense that the induced mappings from a reference $\mathbb{RP}^2$ with $4$ fixed points in general position to $\xi^3(t) \subset \mathbb{RP}^3$ vary continuously.

Throughout the following, we shall once and for all fix such a normalization and view all domains ${\text{dev}(t)}$ as subsets of $\mathbb{RP}^2 \cong \xi^3(t_0)$. When relevant, we will write the map $\xi^3(t) \to \xi^3(t_0)$ by $N_{t \to t_0}$. We denote $N_{t \to t_0}(\text{dev}(t))$ by $C_t$. At times when not doing so would make notation extremely cumbersome, we abuse notation to suppress the normalization used to identify $\text{dev}(t)$ and $C_t$.

\begin{definition}
    Given a Hitchin representation $\rho$, domains $C_t$ as above, a subset $S \subset \partial \Gamma$, and a reference point $s_0 \in S$, a {\rm{projective equivalence of leaves over}} $S$ is a function {\rm{$f: S \to \text{Aut}(\xi^3(t_0))$}} so that $f(t)C_{s_0} = C_t$ for all $t \in S$.
\end{definition}

Projective equivalences of leaves need not exist over a given subset $S \subset \partial \Gamma$. The leaf map $\mathfrak{s}_\rho$ is constant if and only if a family of projective equivalences over $\partial \Gamma$ exists. We do not assume continuity or any sort of regularity, measurability, or the like of projective equivalences over sets $S$ unless explicitly noted.

At times, it will be useful to consider projective equivalences of leaves as two-argument maps between leaves seen as subsets of $\mathbb{RP}^3$, which the next bit of notation facilitates.

\begin{definition}
    Given a projective equivalence $f$ of leaves over $S$ and $t, t' \in S$, define the projective equivalence $f(t,t'): \text{dev}(t) \to \text{dev}(t')$ by $$f(t,t') = N_{t'\to t_0}^{-1} \circ f(t') \circ f(t)^{-1} \circ N_{t \to t_0}.$$ 
\end{definition}

We adopt one final piece of notation in the following: if $x \in \partial \Gamma$ and $p \in \partial C_x$, if $x$ is $\alpha$-modelled, we denote the modelling coefficient $\alpha$ of $C_x$ at $p$ by $\text{model}_x(p)$.

\subsection{Outline of Proof that non-Fuchsian Leaf Maps are Nonconstant}\label{outline-of-main-proof}

Our proof assumes that $\mathfrak{s}_\rho$ is constant, so that there is a projective equivalence $f$ over $\partial \Gamma$, and proves that $\rho$ is $4$-Fuchsian through obtaining constraints on the eigenvalues of $\rho(\Gamma)$.

In order to get initial leverage for our arguments, we require some control on the automorphisms of individual leaves $\mathfrak{s}_\rho(x)$. The dichotomy we use to get this control is the closed subgroup Theorem, which in our setting implies that either for every $x\in \partial \Gamma$ every $\mathfrak{s}_\rho(x)$ has discrete projective automorphism group, or there is an $x \in \partial \Gamma$ so that $\text{Aut}(\xi^3(t_0), C_x) \subset \text{SL}(3, \mathbb{R})$ contains a $1$-parameter subgroup.

The discrete case is the most involved. In it, we first show that though $f$ may be everywhere discontinuous, we may modify $f$ to obtain a \textit{continuous} family $\widetilde{f}$ of projective equivalences over a nonempty open set $U \subset \partial \Gamma$, which can be enlarged using equivariance of leaf maps. The informal idea of the phenomenon underlying why this possible is that all of the discontinuity of $f$ comes from two sources: projective automorphisms of $\mathfrak{s}_\rho(x)$, and divergent families of projective equivalences $A_t$ so that $A_t \overline{C_{t_0}}$ converges to $\overline{C_{t'}}$ in the Hausdorff topology for some $t'$. This is exploited by carefully choosing countable covers $S_i$ of $\partial \Gamma$ so that $f$ is well-behaved on each $S_i$, then applying the Baire category theorem to show some $S_i$ is large enough to be useful.

Next, we use a ``sliding'' argument based on $\alpha$-modelling of boundary points to show that if $\gamma \in \Gamma$ and there is a continuous family of projective equivalences $g$ over an appropriate open set $U_\gamma \subset \partial \Gamma$, the logarithms of the eigenvalues of $\rho(\gamma)$ satisfy a homogeneous polynomial.

Finally, we apply the eigenvalue constraints obtained from the condition that $\mathfrak{s}_\rho$ is constant to show that $\rho$ must be Fuchsian. We use two tools here. Our starting point is that the classification of Zariski closures of Hitchin representations forces $\rho(\Gamma)$ to be Zariski dense in an appropriate Lie group. Inside this Lie group, we may apply work of Benoist on limit cones of Zariski-dense subgroups and find that our polynomial constraint is incompatible with the structure of limit cones unless $\rho$ is Fuchsian.

If any leaf has non-discrete automorphism group, the closed subgroup theorem forces any leaf to have extremely restricted structure, and in particular a rather smooth boundary. This, together with the closedness of the image of $\mathfrak{s}_\rho$, reduces to the case where every leaf is an ellipse. This is then handled with the classification of Zariski closures of Hitchin representations and Benoist's limit cone theorem, as in the discrete case.

The discrete case is the topic of \S\ref{subsection-discrete-case}. Continuity is addressed in \S \ref{subsubsection-discrete-continuity} and boundary $\alpha$-modelling constraints in \S \ref{subsubsection-discrete-continuity-to-eigenvalues}. In \S \ref{subsubsection-zariski-closures} we show that $\rho$ is Fuchsian or $\rho(\Gamma)$ is Zariski dense. In \ref{subsubsection-Limit-Cone} we recall Benoist's theorem on limit cones and apply it to show that $\rho$ is Fuchsian. We show that leaf maps have closed image in \ref{section-thanks-Yves}. The non-discrete case is then completed in \S \ref{subsection-non-discrete}. We explain how Theorems \ref{theorem-fuchsian-locus-characterization} and \ref{headline-theorem-wild-pathology} follow in \S \ref{subsubsection-putting-it-together}.

\subsection{The Discrete Case}\label{subsection-discrete-case} In this subsection, we assume that the group $\text{Aut}(\xi^3(t_0), C_x)$ of projective automorphisms of $C_x$ is discrete for all $x \in \partial \Gamma$.

\subsubsection{Continuity}\label{subsubsection-discrete-continuity} We contend first with the poor separation of points in $\mathfrak{C}$. Some intuition from Benz\'ecri's compactness theorem is that for a domain $\Omega$ with $\text{Aut}(\Omega)$ discrete, projective equivalences of $\text{Aut}(\Omega)$ and divergent sequences $A_n$ so $A_n \Omega \to \Omega$ in $\mathfrak{C}$ should be the only possible discontinuities of a family of projective equivalences. The key observation of this paragraph is that in this setting, as in these two examples, all of the discontinuity of $f$ comes from jumps of (locally) definite size.

It is useful to know that the domains $C_t$ vary continuously in the Hausdorff topology.

\begin{lemma}[Leaf Map Basics]\label{lemma-leaf-maps-basics}
    Let {\rm{$\rho \in \text{Hit}_4(S)$}}. Then $C_t$ is continuous in $t$, $\mathfrak{s}_\rho$ is continuous, and if $x \in \partial \Gamma$ we have $\mathfrak{s}_\rho(x) = \mathfrak{s}_\rho(\gamma x)$ for all $\gamma \in \Gamma$.
\end{lemma}

Note that orbits of the action of $\Gamma$ on $\partial \Gamma$ are dense, as this action is minimal. So for all $x \in \partial \Gamma$, the leaf map $\mathfrak{s}_\rho(x)$ is constant on the dense set $\Gamma x$.

\begin{proof}
     Observe that $\overline{C_t}$ varies continuously in the Hausdorff topology on domains in $\xi^3(t_0)$, since $\partial C_t$ is parametrized by the continuous function $\partial \Gamma \to \xi^3(t_0)$ given by $N_{t\to t_0} \circ \Xi_{t_0 \to t}(x)$ for $x \in \partial \Gamma$, and $\Xi_{t_0\to t}$ depends continuously on $t$. So $\mathfrak{s}_\rho(t) = [C_t] \in \mathfrak{C}$ varies continuously.
     
     For the other claim, if $\gamma \in \Gamma$ we have $\mathfrak{s}_\rho(\gamma x) = [\rho(\gamma)(\text{dev}(x))]$, where $\rho(\gamma)|_{\xi^3(x)} : \xi^3(x) \to \xi^3(\gamma x)$ is induced by a linear map and hence a projective equivalence.
\end{proof}

We are now ready to prove the main proposition of this paragraph.

\begin{proposition}[Modify to Continuity]\label{key-reduction}
    Suppose that $\mathfrak{s}_\rho$ is has countable image and every leaf $\mathfrak{s}_\rho(x)$ has discrete automorphism group. Then $\mathfrak{s}_\rho$ is constant and there is a continuous projective equivalence $\widetilde{f}$ of leaves over a non-empty open set $U \subset \partial \Gamma$.
\end{proposition}

\begin{proof}
    By hypothesis, we may write $\partial \Gamma = \bigsqcup_{m=1}^\infty D_m$ with $D_m$ sets so that for all $m \in \mathbb{N}$ there is some projective equivalence of leaves $f_m$ over $D_m$ with respect to a reference point $s_m \in D_m$. To begin, let us fix a right-invariant metric $d_P$ on $\text{SL}(3,\mathbb{R})$, and a metric $d_{S}$ on $\partial \Gamma$. Note that for all $s \in D_m$, we have $\text{Aut}(\xi^3(t_0), C_s) = f_m(s)\text{Aut}(\xi^3(t_0), C_{s_m})f_m(s)^{-1}$.
    
    To proceed, we need locally uniform control in $f_m(s)$ on the separation of $\text{Aut}(\xi^3(t_0), C_s)$ from the identity. To this end, we adopt the notation that for $\Lambda$ a discrete subgroup of a Lie group $G$ equipped with a right-invariant metric we set $\kappa(\Lambda) := \text{inf} \{ d(e,g) \mid g \in \Lambda -\{e\}\}$. Let us abbreviate conjugation by $\Psi_g: h \mapsto gh g^{-1}$. We obtain control through the following fact, which is a straightforward consequence of differentiablity of conjugation. We include a proof in the appendix for the convenience of the reader.
    
    \begin{lemma}[Discreteness is Conjugation-Stable]\label{lemma-discreteness-conjugation-stable}
    Let $G$ be a Lie group and $\Lambda < G$ be a discrete subgroup. Consider the function $\eta: g \mapsto \kappa(\Psi_g(\Lambda))$. Let $g_0 \in G$ be given. Then there is a neighborhood $U$ of $g_0$ so that $\eta(h) > \kappa(\Psi_{g_0}(\Lambda))/3$ for all $h \in U$.
\end{lemma}
    
    Let $m$ be given, with reference point $s_m \in D_m$. By Lemma \ref{lemma-discreteness-conjugation-stable} (Discreteneness is Conjugation-Stable), to each $g \in \text{SL}(3,\bbR)$, there exists a set $K_g$ with the following properties:
    \begin{enumerate}
        \item $K_g$ is compact and contains $g$ in its interior,
        \item Letting $\kappa_g$ denote $\inf\limits_{h \in K_g}(\kappa(\Psi_{h^{-1}}(\text{Aut}(\xi^3(t_0), C_{s_m})))) = \inf\limits_{h\in K_g}(\kappa(\text{Aut}(\xi^3(t_0), hC_{s_m}))),$ we have $\kappa_g > 0$,
        \item The map $K_g \times K_g \to \text{SL}(3,\bbR)$ given by $(h_1,h_2) \mapsto h_1h_2^{-1}$ has image contained in the ball $B_{\kappa_g/2}(e)$.
    \end{enumerate}
    
    Now let $\{ K_{g_i}^m\}$ be a countable cover of $\text{SL}(3,\mathbb{R})$ by such compact sets. Define $S_i^m \subset \partial \Gamma$ as $f_m^{-1}(K_{g_i}^m)$. We show:

    \textbf{Claim:} The restriction of $f_m$ to $S_i^m$ is uniformly continuous.

    \begin{proof}[Proof of Claim]
    Fix $\epsilon > 0$. We must exhibit that there is some $\delta > 0$ so that if $d_S(t, t') < \delta$ and $f_m(t), f_m(t') \in K_{g_i}^m$, then $d_P(f_m(t), f_m(t')) < \epsilon$.

    We first remark that the map $\overline{B_{\kappa_{g_i}/2}(e)} \times K_{g_i}^m \to \mathbb{R}$ given by $(A,h) \mapsto d_{\text{Haus}}(h\overline{C_{s_m}}, Ah\overline{C_{s_m}})$ is continuous and has zero set exactly $\{e\} \times K_{g_i}^m$ by construction of $\kappa_{g_i}$. It follows from compactness that there is an $\epsilon'> 0$ so that if $ h \in K_{g_i}^m$, $A \in \overline{B_{\kappa_{g_i}/2}(e)}$, and $d_{\text{Haus}}(h\overline{C_{s_m}}, Ah\overline{C_{s_m}}) < \epsilon'$, then $A \in B_\epsilon(e)$.

    As $\partial \Gamma$ is compact, the map $t \mapsto \overline{C_t}$ is uniformly continuous with respect to the Hausdorff topology on $\xi^3(t_0)$, hence there is a $\delta > 0$ so that if $d_S(t,t') < \delta$, then $d_{\text{Haus}}(\overline{C_t}, \overline{C_{t'}}) < \epsilon'$. So if $d_S(t,t') < \delta$ and $t,t' \in S_i^m$, we have $$\epsilon' > d_{\text{Haus}}(\overline{C_t}, \overline{C_{t'}}) = d_\text{Haus}( \overline{C_t}, f_m(t') f_m(t)^{-1} \overline{C_{t}}).$$ As $C_t = f_m(t) C_{s_m}$ with $f_m(t) \in K_{g_i}^m$ and $f_m(t') f_m(t)^{-1} \in B_{\kappa_{g_i}/2}(e)$, we have from our previous observation that $\epsilon > d_P(e, f_m(t') f_m(t)^{-1}) = d_P(f_m(t'), f_m(t))$ by right-invariance.
    \end{proof}

    The point of this claim to us is that for any $i$ and $m$, there exists a continuous extension $\widetilde{f}_i^m$ of $f_m|_{S_i^m}$ to $\overline{S_i^m}.$ So  $\widetilde{f}_i^m$ is a continuous projective equivalence of leaves over $\overline{S_i^m}$.

    Now, as $S_i^m$ cover $\partial \Gamma$ the collection $\{ \overline{S_i^m}\}$ is a countable cover of $\partial \Gamma$ by closed sets. So by the Baire category theorem at least one $\overline{S_i^m}$ has non-empty interior. For any such $i,m$, setting $\widetilde{f} = \widetilde{f}_i^m$ yields the desired continuous family of projective equivalences of leaves over an open set $U$.

    Having produced $\widetilde{f}$, we observe that in fact all $C_x$ for $x\in U$ are projectively equivalent. Since the action of $\Gamma$ on $\partial \Gamma$ is minimal and acts with North-South dynamics, it then follows that all $C_x$ $(x \in \partial \Gamma)$ are projectively equivalent.
\end{proof}

Using the action of $\Gamma$ on $\partial \Gamma$, we may enlarge the open sets where we have continuous families of projective equivalences.

\begin{corollary}[Enlarge Domains]\label{corollary-enlarge-domains}
    Suppose $\mathfrak{s}_\rho$ is constant and every leaf $\mathfrak{s}_\rho(x)$ has discrete automorphism group. Let $\gamma \in \Gamma - \{e\}$ have attracting and repelling fixed-points $\gamma^+, \gamma^- \in \partial \Gamma,$ respectively. Then there is a connected open set $U$ containing $\gamma^+$ and $\gamma^-$ and a continuous projective equivalence of leaves $f$ over $U$.
\end{corollary}

\begin{proof}
    Proposition \ref{key-reduction} (Modify to Continuity) produces an open set $U \subset \partial \Gamma$ and a continuous projective equivalence of leaves $\widetilde{f}$ over $U$. By equivariance of $\text{dev}$, for any $\eta \in \Gamma$ we have $$C_{\eta x} = N_{\eta x \to t_0}(\text{dev}(\eta x)) = N_{\eta x\to t_0}(\rho(\eta) \text{dev}(x)) = N_{\eta x \to t_0} (\rho(\eta)(N_{x \to t_0}^{-1}(C_x))).$$ So defining $f: \eta U \to \text{SL}(3,\bbR)$ by $$\eta x \mapsto N_{\eta x \to t_0} \circ \rho(\eta) \circ N_{x \to t_0}^{-1} \circ \widetilde{f}(x) $$ gives a continuous projective equivalence of leaves over $\eta U$. The corollary now follows from North-South dynamics of the action of $\Gamma$ on $\partial \Gamma$.
\end{proof}

\subsubsection{Boundary Models}\label{subsubsection-discrete-continuity-to-eigenvalues} Throughout this paragraph, we suppress uses of normalization maps $N_{x \to t_0}: \text{dev}(x) \to \xi^3(t_0)$ to make notation manageable. The goal of this paragraph is to prove the following claim.

\begin{proposition}[Modelling Constraints]\label{I-was-dumb}
Suppose that $\rho$ is a Hitchin representation, $\mathfrak{s}_\rho$ is constant, and $\mathfrak{s}_\rho(x)$ has discrete automorphism group for all $x \in \partial \Gamma$. Then for all $\gamma \in \Gamma - \{e\}$, {\rm{$$\text{model}_{\gamma^+}(\xi^2(\gamma^-) \cap \xi^3(\gamma^+)) = \text{model}_{\gamma^-} \xi^1(\gamma^-).$$}} 
\end{proposition}

A key input to our proof of Proposition \ref{I-was-dumb} is the following application of discreteness of automorphism groups of leaves, which allows us to determine the values of a continuous projective equivalence of leaves at specific points. It says that at specific points, continuous projective equivalences of leaves commute with $\rho$ in an appropriate sense.

\begin{lemma}[Commutativity Lemma]\label{lemma-equivalences-commute-well}
    Let $\gamma \in \Gamma - \{e\}$. If $f$ is a continuous projective equivalence of leaves over a connected open set $U$ containing $\gamma^+$ for some $\gamma \in \Gamma - \{e\}$, then for all $s \in U$ and {\rm{$p \in \overline{\text{dev}(\gamma^+)}$}}, we have $$\rho(\gamma)(p) =  [f(\gamma s, \gamma^+) \circ \rho(\gamma) \circ  f(\gamma^+, s)] (p)$$
\end{lemma}

\begin{proof}
The maps $\{A_s\}_{s \in U}$ given by \begin{align*} A_s : \text{dev}(\gamma^+) &\to \text{dev}(\gamma^+) \\ p &\mapsto [f(\gamma s, \gamma^+)\circ \rho(\gamma) \circ f(\gamma^+, s)](p)\end{align*} are a continuous family of projective equivalences of $\text{dev}(\gamma^+),$ and hence must be constant by discreteness of $\text{Aut}(\xi^3(t_0), C_{\gamma^+})$. At $s = \gamma^+$ we have $A_s = \rho(\gamma)$.
\end{proof}

\begin{figure}[t]\label{figure-two-slices-dynamics}
\vspace{-1cm}
\begin{center} \makebox[0pt]{\, \, \,\, \, \, \, \, \includegraphics[scale=0.23]{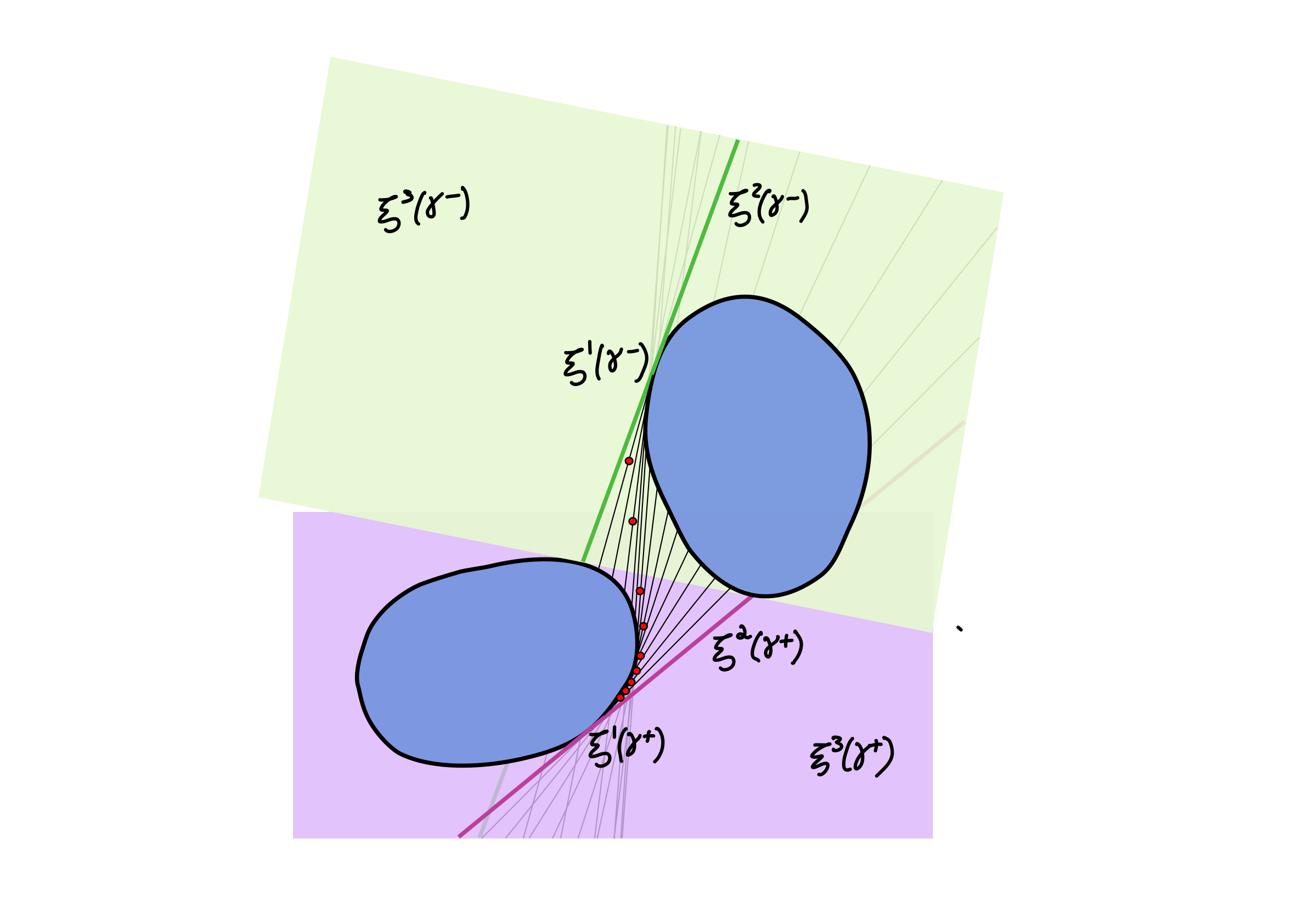}}
\end{center}\vspace{-1.2cm}
\caption{Sketch of the situation of Proposition \ref{I-was-dumb}. The red dots illustrate that if $p = f(\gamma^+,s)(\xi^2(\gamma^-) \cap \xi^3(\gamma^+))$ is not on $\xi^2(\gamma^-)$, then $\rho(\gamma)^np$ converges to $\xi^1(\gamma^+)$.}
\end{figure}

We are now prepared to prove Proposition \ref{I-was-dumb}.

\begin{proof}[Proof of Proposition \ref{I-was-dumb}]
    Let $\gamma \in \Gamma - \{e\}$ be given. By Corollary \ref{corollary-enlarge-domains} there is a connected open set $U$ containing $\gamma^+$ and $\gamma^-$ and a continuous projective equivalence of leaves $f$ over $U$. Let $I \subset \partial \Gamma$ be a closed interval with endpoints $\gamma^+, \gamma^-$. Our strategy is to constrain $f(\gamma^+, s)$ at $\xi^2(\gamma^{-}) \cap \xi^3(\gamma^+)$, then conclude an equality of modelling coefficients at controlled points.

    \textbf{Claim (Stuck to $\xi^2(\gamma^-)$).} For all $s \in I - \{\gamma^-\}$, we have $f(\gamma^+, s)(\xi^2(\gamma^-) \cap \xi^3(\gamma^+)) = \xi^2(\gamma^-) \cap \xi^3(s)$.

    \begin{proof}[Proof of Claim] We compute an auxiliary limit in two different ways. Fix $s \in I - \{\gamma^-\}$. By Lemma \ref{lemma-equivalences-commute-well}, for all $n \in \mathbb{N}$, we have $$ f(\gamma^n s, \gamma^+) \circ \rho(\gamma^n) \big[ f(\gamma^+, s)[(\xi^2(\gamma^{-} )\cap \xi^3(\gamma^+)]\big] = \rho(\gamma)^n [\xi^2(\gamma^{-}) \cap \xi^3(\gamma^+)] = \xi^2(\gamma^{-}) \cap \xi^3(\gamma^+),$$ so that $$\lim_{n \to \infty} f(\gamma^n s, \gamma^+) \circ \rho(\gamma^n) \circ f(\gamma^+, s)[(\xi^2(\gamma^{-} )\cap \xi^3(\gamma^+)] = \xi^2(\gamma^{-}) \cap \xi^3(\gamma^+).$$
    
    On the other hand, let $p \in \partial \text{dev}(s) - \{ \xi^2(\gamma^-) \cap \xi^3(s)\}$. Then $p = \xi^2(t') \cap \xi^3(s)$ for some $t' \neq \gamma^{-}$ or $p = \xi^1(s)$ with $s \neq \gamma^-$. In the first case, we then have (see Figure \ref{figure-two-slices-dynamics}) \begin{align*} \lim_{n \to \infty} \gamma^n s &= \gamma^+ \\ 
    \lim_{n \to \infty}  \rho(\gamma)^n p &= \lim_{n \to \infty} \xi^3(\gamma^n s) \cap \xi^2(\gamma^n t') = \xi^1(\gamma^+), \end{align*} where we have used North-South dynamics of the action of $\Gamma$ on $\partial \Gamma$ and dual osculation. In the second case, we similarly have $$\lim_{n \to \infty} \rho(\gamma)^n \xi^1(s) = \lim_{n \to \infty} \xi^1(\gamma^n s) = \xi^1 (\gamma^+).$$ So by continuity of $f$ and that $f(\gamma^+, \gamma^+)$ is the identity on $\xi^3(\gamma^+),$
    $$\lim_{n \to \infty} f(\gamma^n s, \gamma^+) \circ \rho(\gamma^n) (p) = \xi^1(\gamma^+).$$ 
    As $\xi^2(\gamma^-) \cap \xi^3(\gamma^+) \neq \xi^1(\gamma^+)$, the only possibility is that $f(\gamma^+, s)(\xi^2(\gamma^-) \cap \xi^3(\gamma^+)) = \xi^2(\gamma^-) \cap \xi^3(s)$.
\end{proof}

We next observe that $f(\gamma^-, \gamma^+)(\xi^1(\gamma^-)) = \xi^2(\gamma^-) \cap \xi^3(\gamma^+)$. To see this, note that by dual osculation, for any sequence $s_n \to \gamma^-$ with $s_n \neq \gamma^-$ for all $n$, we have $$\lim\limits_{n \to \infty} \xi^2(\gamma^-) \cap \xi^3(s_n) = \xi^1(\gamma^-).$$ Since $f$ is continuous and $f(s_n, \gamma^+)(\xi^2(\gamma^-) \cap \xi^3(s_n)) = \xi^2(\gamma^-) \cap \xi^3(\gamma^+)$ for all $n$, we must have $f(\gamma^-, \gamma^+)(\xi^1(\gamma^-)) = \xi^2(\gamma^-) \cap \xi^3(\gamma^+)$.
    
Since $f(\gamma^-, \gamma^+): \xi^3(\gamma^-) \to \xi^3(\gamma^+)$ is a projective equivalence sending $\xi^1(\gamma^-)$ to $\xi^2(\gamma^-) \cap \xi^3(\gamma^+)$, we conclude that $\text{model}_{\gamma^+}(\xi^2(\gamma^-) \cap \xi^3(\gamma^+)) = \text{model}_{\gamma^-} \xi^1(\gamma^-)$.
\end{proof}

\subsubsection{Zariski Density}\label{subsubsection-zariski-closures} 
In the following, we fix a lift of $\rho: \Gamma \to \text{PSL}(4,\bbR)$ to $\text{SL}(4,\bbR)$. Such lifts always exist, as mentioned in \S \ref{convex-projective-structures-background}. In this paragraph, we examine the contraints on eigenvalues of $\rho$ given by Proposition \ref{I-was-dumb} (Modelling Constraints) through Lemma \ref{regularity-lemma} (Models at Fixed Points), and prove a dichotomy for the Zariski closure of $\rho$.

We obtain eigenvalue data as follows. Under the hypotheses of Proposition \ref{I-was-dumb}, let $\gamma \in \Gamma - \{e\}$, write the eigenvalues of $\rho(\gamma)$ as $\lambda_1, \lambda_2, \lambda_3, \lambda_4$ ordered by decreasing modulus, and denote $\log |\lambda_i|$ by $\ell_i$ for $i=1,...,4$. As a consequence of $\rho$ being Hitchin, $\ell_1 > \ell_2 > \ell_3 > \ell_4$ and all $\lambda_i$ have the same sign. Denote the corresponding eigenlines by $e_1, e_2, e_3, e_4$. We have $e_1 = \xi^1(\gamma^+)$, $e_2 = \xi^2(\gamma^+) \cap\xi^3(\gamma^-)$, $e_3 = \xi^2(\gamma^-) \cap \xi^3(\gamma^+)$, $e_4 = \xi^1(\gamma^-)$ (see \cite{guichard2008convex} \S 5, in particular Fig. 7 there). Applying Lemma \ref{regularity-lemma} (Models at Fixed Points) to the restrictions of $\rho(\gamma)$ to the invariant subspaces $\xi^3(\gamma^+)$ and $\xi^3(\gamma^-)$ and the constraint $\text{model}_{\gamma^+} (\xi^2(\gamma^-) \cap \xi^3(\gamma^+) ) = \text{model}_{\gamma^-} \xi^1(\gamma^-)$ shows \begin{align}\label{eigenvalue-condition-from-regularity}
\frac{\ell_1 - \ell_3}{\ell_2 - \ell_3} = \frac{\ell_2 - \ell_4}{\ell_3 - \ell_4},
\end{align} or equivalently that
\begin{align}\label{eigenvalue-constraint-homogeneous-form}
(\ell_1 -\ell_3)(\ell_3 - \ell_4) - (\ell_2 - \ell_4)(\ell_2 - \ell_3) = 0.
\end{align}

\begin{remark}
    The homogeneity of Equation \ref{eigenvalue-constraint-homogeneous-form} is responsible for much of the usefulness of this constraint. It is expected since the points where models are computed for $\gamma$ and $\gamma^n$ ($n \in \mathbb{N}$) are the same.
\end{remark} 

\begin{remark}
    One may also apply the same argument to $\gamma^{-1}$ in place of $\gamma$, which establishes an equality of modelling coefficients between two different points than the argument for $\gamma$. The equation so obtained appears distinct from Equation \ref{eigenvalue-condition-from-regularity} at a glance, but the two may be shown to be equivalent. So this offers no new information.
\end{remark}

\begin{figure}
    \centering
    \includegraphics[scale=0.66]{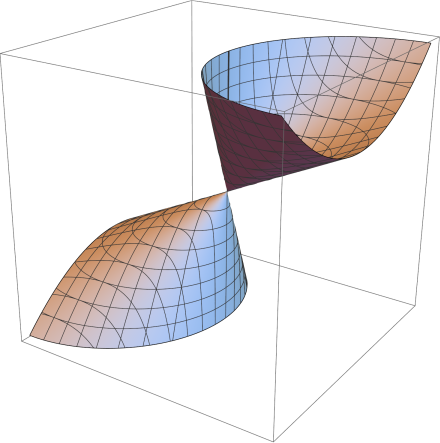}
    \caption{The zero locus of Equation \ref{eigenvalue-constraint-homogeneous-form}. Image generated by Wolfram Mathematica.}
    \label{zero-locus-picture}
\end{figure}

Zariski closures of Hitchin representations have been classified.\footnote{The classification is due to Guichard in unpublished work, and also follows from recent results of Sambarino \cite{sambarino2020infinitesimal}.} For a lift of $\rho$ in the $\text{PSL}(4,\bbR)$ Hitchin component to $\text{SL}(4,\bbR)$, the classification states that the Zariski closure of $\rho(\Gamma)$ is conjugate to a principal $\text{SL}(2,\bbR)$ (in which case $\rho$ is Fuchsian), is conjugate to $\text{Sp}(4,\bbR)$, or is $\text{SL}(4,\bbR)$. We shall show that $\rho$ is Fuchsian through this condition.

We begin by showing that the Zariski closure of $\rho(\Gamma)$ is not conjugate to  $\text{Sp}(4,\bbR)$. The linear algebra behind this case is contained in the next lemma.

\begin{lemma}[Diagonal Form]\label{symplectic-matrix-case}
    Suppose that {\rm{$A \in \text{SL}(4, \mathbb{R})$}} is diagonalizable, with real eigenvalues $(\lambda_1, \lambda_2, \lambda_3,\lambda_4)$ with $|\lambda_1| > |\lambda_2| > |\lambda_3| > |\lambda_4| > 0$, satisfies Equation \ref{eigenvalue-condition-from-regularity}, and $A$ is conjugate to a matrix in {\rm{$\text{Sp}(4,\bbR)$}}. Then $A$ is conjugate to a matrix of the form {\rm$\text{diag}(\lambda^3, \lambda, \lambda^{-1},\lambda^{-3})$} for some $\lambda \in \mathbb{R} - [-1,1]$.
\end{lemma}

\begin{proof}
    This is a computation. Write $\alpha = \ell_1 - \ell_2, \beta = \ell_2 - \ell_3, \gamma = \ell_3 - \ell_4$. Then $\alpha, \beta, \gamma > 0$, and Equation \ref{eigenvalue-condition-from-regularity} is equivalent to $(\alpha + \beta)\gamma = (\beta + \gamma)\beta$, which reduces to $\beta^2 = \alpha \gamma$.
    
    Eigenvalues of semisimple symplectic matricies $A$ come in inverse pairs, i.e. if $\lambda$ is an eigenvalue of $A$ with multiplicity $m$, then $1/\lambda$ is also an eigenvalue with multiplicity $m$. For us, this means that $\ell_1 - \ell_2 = \ell_3 - \ell_4$, so that $\alpha^2 = \beta^2 = \gamma^2$, and by positivity $\alpha = \beta = \gamma$. That $A \in \text{SL}(4, \mathbb{R})$ is to say $\ell_1 + \ell_2 + \ell_3 + \ell_4 = 0$, which implies the claim.
\end{proof}

Any Fuchsian representation $\rho$ has, up to negation, $\rho(\gamma)$ ($\gamma \in \Gamma)$ conjugate to a matrix of the form $\text{diag}(\lambda^3, \lambda, \lambda^{-1}, \lambda^{-3})$ for some $\lambda > 1$. We next show that this property distinguishes Fuchsian representations. In particular, it is not possible for a non-Fuchsian representation to take values in a collection of distinct principal $\text{SL}(2, \mathbb{R})$ subgroups of $\text{SL}(4, \mathbb{R})$.

\begin{proposition}[Fuchsian from Eigenvalues]\label{eigenvalues-that-look-fuchsian-implies-fuchsian}
    Suppose that $\rho$ is lift of a {\rm{$\text{PSL}(4,\bbR)$}} Hitchin representation to {\rm{$\text{SL}(4,\bbR)$}} so that for all $\gamma \in \Gamma$, $\rho(\gamma)$ is conjugate to a matrix of the form $\text{diag}(\lambda^3, \lambda, \lambda^{-1}, \lambda^{-3})$ for some positive $\lambda = \lambda(\gamma) \in \mathbb{R} - \{0\}$. Then $\rho$ is Fuchsian.
\end{proposition}

\begin{remark}
A shorter proof of the below is possible using the theorem of Benoist described in the next paragraph. The below proof is included due to its explicitness and its lack of direct reliance on such heavy machinery: in place of of Benoist's limit cone theorem, it uses the fundamental theorem of symmetric polynomials.
\end{remark}

\begin{proof}
     By the classification of Zariski closures of Hitchin representations \cite{sambarino2020infinitesimal}, it suffices to show that the Zariski closure of $\rho(\Gamma)$ is neither $\text{SL}(4,\mathbb{R})$ nor conjugate to $\text{Sp}(4, \mathbb{R})$.
     
    We begin by recalling that if $a_1, ..., a_4$ are the eigenvalues of $A \in \text{GL}(4,\mathbb{R})$, then the coefficients $\sigma_i$ ($i = 0,...,3$) of the characteristic polynomial of $A$ are the elementary symmetric polynomials in the variables $a_1,...,a_4$, and are all polynomials in the entries of $A$. So let $F(a_1,a_2,a_3,a_4) = \prod_{i,j \in \{1,...,4\}} (a_i - a_j^3)$. Then $F$ is a symmetric polynomial in $\{a_1,...,a_4\}$, and so is an element of the polynomial ring $\mathbb{Z}[\sigma_0,...,\sigma_3]$ by the fundamental theorem of symmetric polynomials. Consequently, $F$ is a polynomial $G$ in the entries of $A$. As all $\sigma_i$ are conjugation-invariant, so is $G$.

    Note, furthermore, that if $A$ is conjugate to a matrix of the form $\text{diag}(\lambda^3, \lambda, \lambda^{-1}, \lambda^{-3})$, then $F(\lambda^3, \lambda, \lambda^{-1}, \lambda^{-3})$ vanishes. So for a Hitchin representation $\rho$ satisfying our hypotheses, the Zariski closure of $\rho(\Gamma)$ is contained in the vanishing locus of $G$.

    On the other hand, for instance, the symplectic matrix $A = \text{diag}(3, 2, 1/2, 1/3) \in \text{Sp}(4, \mathbb{R})$ is not in the vanishing locus of $G$, as $F(3,2,1/2,1/3) \neq 0$. As $G$ is conjugation-invariant, this shows that the Zariski closure of $\rho(\Gamma)$ cannot contain any subgroup of $\text{SL}(4, \mathbb{R})$ conjugate to $\text{Sp}(4, \mathbb{R})$, which gives the claim. \end{proof}

We immediately obtain:

\begin{corollary}[Zariski Closure Dichotomy]\label{proposition-zariski-dichotomy}
    Suppose that for all $\gamma \in \Gamma - \{e\}$, we have {\rm{$\text{model}_{\gamma^+} (\xi^2(\gamma^-) \cap \xi^3(\gamma^+) ) = \text{model}_{\gamma^-} \xi^1(\gamma^-)$}}. Then $\rho$ is $4$-Fuchsian or Zariski dense.   
\end{corollary}

\begin{proof}
    Combine Lemma \ref{symplectic-matrix-case} and Proposition \ref{eigenvalues-that-look-fuchsian-implies-fuchsian} and the classification of Zariski closures of Hitchin representations.
\end{proof}

\subsubsection{Limit Cones}\label{subsubsection-Limit-Cone} We finish the discrete case here by showing:

\begin{proposition}[Zariski Density Impossible]\label{proposition-no-zariski-density}
    Suppose that $\mathfrak{s}_\rho$ is constant and $\mathfrak{s}_\rho(x)$ has discrete automorphism group. Then $\rho(\Gamma)$ can not be Zariski-dense in {\rm{$\text{SL}(4,\bbR)$}}.
\end{proposition}

The source of our obstruction is an incompatibility of the eigenvalues of $\rho$ with Zariski density. The perspective we take to demonstrate the incompatibility is to analyze the \textit{limit cone} $\ell_{\rho(\Gamma)}$, which has been studied for Zariski-dense representations by Benoist. We begin by recalling the definition of limit cones and Benoist's theorem. The relevant theory has been developed for connected real reductive linear semisimple Lie groups $G$, but we shall deal exclusively with the cases $G = \text{SL}(4,\bbR)$ and $\text{Sp}(4,\bbR)$.

Let $H$ be a Cartan subgroup of $\text{SL}(4,\bbR)$, e.g. the diagonal matricies of determinant $1$ with respect to a choice of basis of $\bbR^4$, and $\mathfrak{a}_{\mathfrak{sl}(4,\bbR)}$ the corresponding Cartan subalgebra of $\mathfrak{sl}(4,\bbR)$. We identify $\mathfrak{a}_{\mathfrak{sl}(4,\bbR)}$ with the hyperplane $$\{ (x_1, x_2, x_3, x_4) \in \bbR^4 \mid x_1 + x_2 + x_3 + x_4 = 0 \} $$
and take the closed Weyl chamber $\mathfrak{a}_{\mathfrak{sl}(4,\bbR)}^+ \subset \mathfrak{a}_{\mathfrak{sl}(4,\bbR)}$ given by $$\{(x_1, x_2, x_3,x_4) \in \mathfrak{a}_{\mathfrak{sl}(4,\bbR)} \mid x_1 \geq x_2 \geq x_3 \geq x_4 \}.$$

Let $H_{\text{Sp}} < H$ be a Cartan subgroup of $\text{Sp}(4,\bbR)$, e.g. the elements of $H$ preserving the standard symplectic form, with corresponding Cartan subalgebra $\mathfrak{a}_{\mathfrak{sp}(4,\bbR)}$ identified with the elements $(x_1, x_2, x_3, x_4) \in \mathfrak{a}_{\mathfrak{sl}(4,\bbR)}$ with $x_1 + x_4 = x_2 + x_3 = 0$.

For $A \in \text{SL}(4,\bbR)$, and $i =1,2,3,4$ denote by $\lambda_i(A)$ the generalized eigenvalue of $A$ with $i^{\text{th}}$ largest modulus. We define \begin{align*}
    \Lambda: \text{SL}(4,\bbR) &\to \mathfrak{a}^+_{\mathfrak{sl}(4,\bbR)} \\ A &\mapsto (\log |\lambda_1(A)|, \log |\lambda_2(A)|, \log |\lambda_3(A)|, \log |\lambda_4(A)|).
\end{align*} 

\begin{definition}
    Given a subgroup {\rm{$H < \text{SL}(4,\bbR)$}} (resp. {\rm{$\text{Sp}(4,\bbR)$}}), the {\rm{{limit cone}}} $\ell_H$ of H is the smallest closed cone in $\mathfrak{a}_{\mathfrak{sl}(4,\bbR)}^+$ (resp. $\mathfrak{a}_{\mathfrak{sp}(4,\bbR)}^+$) containing $\Lambda(H)$.
\end{definition} For us, $\ell_{\rho(\Gamma)}$ is the closure of the half-lines spanned by $(\ell_1, \ell_2, \ell_3, \ell_4)$ in the notation of the previous section. The following is due to Benoist:
\begin{theorem}[Benoist \cite{benoist1997proprietes}]\label{benoist-limit-cone-theorem}
    Suppose {\rm{$H < \text{SL}(4,\bbR)$}} (resp. {\rm{$H < \text{Sp}(4,\bbR)$}}) is Zariski dense. Then $\ell_H$ is a convex cone with nonempty interior in $\mathfrak{a}_{\mathfrak{sl}(4,\bbR)}^+$ (resp. $\mathfrak{a}_{\mathfrak{sp}(4,\bbR)}^+$).
\end{theorem}

In fact, Benoist proved much more in \cite{benoist1997proprietes}, such as realizability of convex cones with nonempty interior by Zariski-dense subgroups and equivalence of $\ell_H$ and an analogous definition in terms of singular values. The above is what we need.

We are now ready to complete the discrete case.

\begin{proof}[Proof of Proposition \ref{proposition-no-zariski-density}] For any $\gamma \in \Gamma$, the logarithms $(\ell_1, \ell_2, \ell_3, \ell_4)$ of the absolute values of the eigenvalues of $\rho(\gamma)$ must satisfy the \textit{homogeneous} degree $2$ polynomial of Equation \ref{eigenvalue-constraint-homogeneous-form}: $F(x_1,x_2,x_3,x_4) = (x_1 - x_3)(x_3 - x_4) - (x_2 - x_4)(x_2 - x_3) = 0$. This polynomial is not uniformly $0$ on $\mathfrak{a}^+_{\mathfrak{sl}(4,\bbR)}$, and so by homogeneity has zero set $X$ that is a closed cone of positive codimension. As $\Lambda(\rho(\Gamma)) \subset X$, the limit cone $\ell_{\rho(\Gamma)}$ must have empty interior, which by Benoist's theorem is impossible if $\rho(\Gamma)$ is Zariski-dense.
\end{proof}

\subsection{The Collection of all Leaves}\label{section-thanks-Yves} 
 Following a suggestion of Benoist, we adapt an argument of Benz\'ecri \cite{benzecri1960thesis} (see also \cite{goldman2022geometric}, proof of Theorem 4.5.6) to characterize the image of $\mathfrak{s}_\rho$. This will be used to give a short proof in the case of non-discrete automorphism group below. We maintain the notations of the previous section, notably the normalization $N$.
 
 \begin{proposition}[Benoist]\label{image-closed}
    Let $t \in \partial \Gamma$. Then $\mathfrak{s}_\rho(\partial \Gamma) = {\rm{Cl}}_\mathfrak{C}(\{[C_t]\})$.
\end{proposition}

\begin{proof}
    From the minimality of the action of $\Gamma$ on $\partial \Gamma$, the continuity $\mathfrak{s}_\rho$, and the observation that for $\gamma \in \Gamma$ and $t \in \partial \Gamma$ that $[\mathfrak{s}_\rho(t)] = [\mathfrak{s}_\rho(\gamma t)]$ we see that $\mathfrak{s}_\rho(\partial \Gamma) \subset \text{Cl}_{\mathfrak{C}}(\{C_t\})$. So it suffices to show that $\mathfrak{s}_\rho(\partial \Gamma)$ is closed in $\mathfrak{C}$.
    
    We next describe the condition we shall verify in order to prove this. To show $\mathfrak{s}_\rho(\partial \Gamma)$ is closed in $\mathfrak{C}$ it suffices to show that the union of the $\text{SL}(3,\mathbb{R})$-orbits of $\{C_t \}$ $(t \in \partial \Gamma)$ is closed in $\mathcal{C}$, which is equivalent to the closedness of the union of the $\text{SL}(3,\bbR)$-orbits of the preimages $\Pi^{-1}(\{\mathfrak{s}_\rho(t) \}) = \{\mathfrak{s}_\rho(t) \} \times \mathfrak{s}_\rho(t)$ $(t \in \partial \Gamma)$ in $\mathcal{C}^*$. This is in turn equivalent to showing the image $\mathfrak{L}$ of $\bigcup_{t\in \partial \Gamma} \Pi^{-1}(\{\mathfrak{s}_\rho(t) \})$ under the projection $\mathcal{Q}^*: \mathcal{C}^* \to \mathfrak{C}^*$ is closed. By Benz\'ecri's compactness theorem, $\mathfrak{C}^*$ is a Hausdorff space and so compact sets in $\mathfrak{C}^*$ are closed. As $\mathfrak{C}^*$ is second-countable, it suffices to verify that $\mathfrak{L}$ is sequentially compact. This is what we shall prove.

    Fix a compact set $K \subset \Omega_\rho$ so that $\rho(\Gamma) K = \Omega_\rho$. One verifies using compactness of $K$ and $\partial \Gamma$ that the image of $K$ after normalization is uniformly separated from the complement of the leaves in the sense that there is some $\delta > 0$, independent of $t \in \partial \Gamma$, so that if $p \in K \cap \xi^3(t)$ then $d_{t_0}(N_t(p), C_t^c) > \delta$.

    So let $c_n \in \mathfrak{L}$ be a sequence. For all $n$, choose a leaf $C_{t_n}$ and $p_n \in \text{Int}(C_{t_n})$ so that $\mathcal{Q}^*((C_{t_n}, p_n)) = c_n$. Since $\rho(\Gamma)K = \Omega_\rho$, after applying projective equivalences arising from compositions of normalizations and the action of $\rho(\gamma)$ ($\gamma \in \Gamma$) on $\Omega_\rho$, we may arrange for $p_n \in K \cap C_{t_n}$. It follows from compactness of $K$ and continuity features of our normalization that after taking a subsequence, there is some $t_\infty \in \partial \Gamma$ and $p_\infty \in C_{t_\infty} \cap K$ so that $\lim\limits_{n \to \infty} (C_{t_n}, p_n) = (C_{t_\infty}, p_\infty)$ in $\mathcal{C}^*$. Hence $\mathcal{Q}^*(C_{t_\infty}, p_\infty) \in \mathfrak{L}$ is a limit point of $c_n$ and so $\mathfrak{L}$ is compact, as desired.
\end{proof}

\subsection{The Non-Discrete Case}\label{subsection-non-discrete}

We show:
\begin{proposition}[Only Ellipses]\label{non-discrete-goal}
    Suppose there is some $x \in \partial \Gamma$ so that $\mathfrak{s}_\rho(x)$ has non-discrete automorphism group. Then $\rho$ is $4$-Fuchsian and $\mathfrak{s}_\rho(y)$ is the ellipse for all $y \in \partial \Gamma$.
\end{proposition}

\begin{figure}[t]
    \centering
    \makebox[0pt]{\includegraphics[scale=0.56]{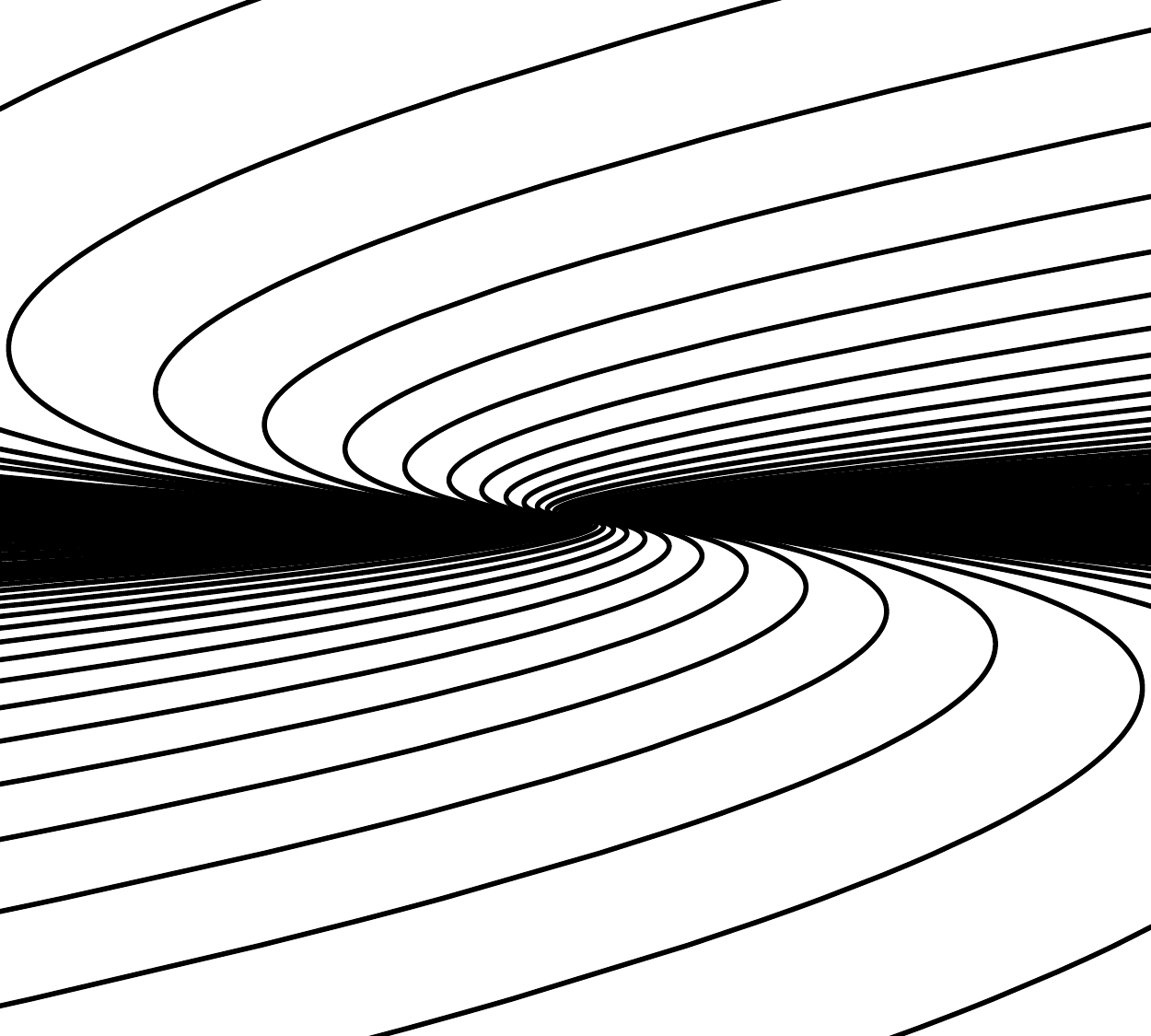} \hspace{-0.5cm} \includegraphics[trim={0 1.1cm 1cm 0}, scale=0.42]{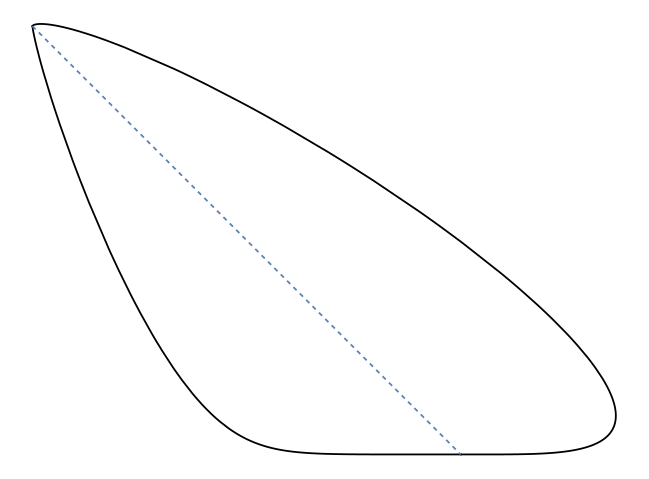}}
    \caption{Some sample orbits of one-parameter subgroups of $\text{SL}(3,\bbR)$.}
    \label{fish-hooks-picture}
\end{figure}  

\begin{proof}
    Let $x \in \partial \Gamma$ be so that $\text{Aut}(\xi^3(t_0), C_x)$ is non-discrete. Then by the closed subgroup theorem, $\text{Aut}(\xi^3(t_0), C_x)$ contains a one-parameter subgroup $H = \{ A_t \}_{t \in \bbR}$. Then for any $p_0 \in \partial C_x$ the orbit $Hp_0$ is entirely contained in $\partial C_x$.
    
    Since fixed-points of $A_t$ for $t \neq 0$ are either isolated or contained in a line of fixed points and $C_x$ is strictly convex, it follows that $\partial C_x$ contains a nontrivial orbit $\mathcal{O}$ of $H$, which must be smooth. Note that $\mathcal{O}$ cannot have everywhere vanishing curvature, since then $\mathcal{O}$ would be a line segment and $C_x$ is strictly convex. So $\partial C_x$ must have a $C^2$ point of nonvanishing curvature. It is then a standard fact (e.g. \cite{goldman2022geometric} Ex. 4.5.2.3) that the ellipse $[O] \in \mathfrak{C}$ is contained in the $\mathfrak{C}$-closure of $\{ [C_x]\}$. 

    By proposition \ref{image-closed}, there is some $y \in \partial \Gamma$ so $[C_y] = [O]$. Since the projective class $[O]$ of the ellipse is a closed point of $\mathfrak{C}$, by Lemma \ref{lemma-leaf-maps-basics} (Leaf Map Basics) the preimage $\mathfrak{s}_\rho^{-1}(\{[O]\}) \subset \partial \Gamma$ is closed and contains a dense subset of $\partial \Gamma$, hence must be all of $\partial \Gamma$. So for all $t \in \partial \Gamma,$ the leaf $C_t$ is an ellipse.
  
  Let $\gamma \in \Gamma - \{e\}$. Since we know $C_{\gamma^+}$ and $C_{\gamma^-}$ are ellipses, we may apply the $\alpha$-modelling Lemma \ref{regularity-lemma} to $\xi^1(\gamma^+)$ and $\xi^1(\gamma^-)$, which yields (in the notation of \S\ref{subsubsection-zariski-closures})

  \begin{align}\label{equation-ellipse-constraints} \frac{\ell_1 - \ell_3}{\ell_1 - \ell_2} =2 = \frac{\ell_2 - \ell_4}{\ell_3 - \ell_4}.\end{align} Write $\alpha = \ell_1 - \ell_2, \beta = \ell_2 - \ell_3, \gamma = \ell_3 - \ell_4$; all are positive.
  
  The first equality of Equation \ref{equation-ellipse-constraints} shows that $\alpha = \beta$, and the second equality $2 = (\beta + \gamma)/\gamma$ shows that $\beta = \gamma$. So $\alpha =\beta = \gamma$, which together with the condition $\ell_1 +\ell_2 + \ell_3 + \ell_4 =0$ (due to $\rho(\gamma) \in \text{SL}(4,\bbR)$) shows that $\rho(\gamma)$ is conjugate to a matrix of the form $\text{diag}(\lambda^3, \lambda, \lambda^{-1}, \lambda^{-3})$ for some $\lambda > 1$. Proposition \ref{eigenvalues-that-look-fuchsian-implies-fuchsian} (Fuchsian from Eigenvalues) shows $\rho$ is Fuchsian. \end{proof}

\subsection{Deduction of Results}\label{subsubsection-putting-it-together}

We end by documenting how the results claimed in the introduction follow. We first note:

\begin{theorem}\label{theorem-nonconstant-leafs}
    The leaf map $\mathfrak{s}_\rho$ is constant if and only if $\rho$ is Fuchsian. If $\rho$ is Fuchsian, then $\mathfrak{s}_\rho$ takes value the ellipse. 
\end{theorem}

\begin{proof}
    The Fuchsian case is shown by Guichard-Wienhard in \cite{guichard2008convex}. That $\mathfrak{s}_\rho$ is not constant if $\rho$ is not Fuchsian follows from Corollary \ref{proposition-zariski-dichotomy}, Proposition \ref{proposition-no-zariski-density}, and Proposition \ref{non-discrete-goal}. 
\end{proof}

The main theorems follow:

\begin{proof}[Proof of Theorem \ref{theorem-fuchsian-locus-characterization}]
    The first equivalence is Theorem \ref{theorem-nonconstant-leafs}. The equivalence of (2) and (3) is given by the equivalence of constancy and countable image in Proposition \ref{key-reduction}. The parts of the equivalence of (4) with (1) pertaining to Fuchsian representations follow from standard facts about ellipses. That a closed point of $\mathfrak{C}$ or a divisible domain occuring as a leaf implies (2) follows from that divisible domains are closed points of $\mathfrak{C}$ together with Lemma \ref{lemma-leaf-maps-basics} (Leaf Map Basics). That a leaf having non-discrete automorphism group implies $\rho$ is Fuchsian follows from Proposition \ref{non-discrete-goal}, in particular the lack of the assumption that $\mathfrak{s}_\rho$ is constant there. 
\end{proof}

\begin{proof}[Proof of Theorem \ref{headline-theorem-wild-pathology}]
    Combine Theorem \ref{theorem-nonconstant-leafs} with Lemma \ref{lemma-leaf-maps-basics} (Leaf Map Basics).
\end{proof}

\begin{proof}[Proof of Theorem \ref{thm-solves-benzecri}]
    For $\rho$ non-Fuchsian, that $\mathfrak{s}_\rho(\partial \Gamma)$ is a non-point closed subset of $\mathfrak{C}$ is Theorem \ref{theorem-fuchsian-locus-characterization} and Proposition \ref{image-closed}. That $\mathfrak{s}_\rho(\partial \Gamma)$ is minimal among closed sets follows from the characterization in Proposition \ref{image-closed} that $\mathfrak{s}_\rho(\partial \Gamma)$ is the closure of any point in $\mathfrak{s}_\rho(\partial \Gamma)$. This proves the Theorem for $\mathfrak{C}(\RP^2)$.

    The result for $\mathfrak{C}(\RP^n)$ for $n \geq 3$ reduces to the $n = 2$ case by a classical characterzation of $\mathfrak{C}(\RP^n)$-closures of convex hulls due to Benz\'ecri \cite{benzecri1960thesis} (\S V.3, Proposition 4). In particular, Benz\'ecri's characterization implies that if $\Omega \in \mathfrak{C}(\RP^n)$ has $\text{Cl}_{\mathfrak{C}(\RP^n)}(\{ [\Omega]\})$ a non-point minimal closed subset of $\mathfrak{C}(\RP^n)$, then the same is true in $\mathfrak{C}(\RP^{n+1})$ of the domain $\Omega' \subset \RP^{n+1}$ formed as follows. Take the convex hull in an affine chart $\mathcal{A}$ of an inclusion of $\Omega$ in the intersection with $\mathcal{A}$ of a copy $P$ of $\RP^n \subset \RP^{n+1}$ and a point $p \in \mathcal{A} - P$.
\end{proof}

\begin{proof}[Proof of Corollary \ref{cor-arbitrary-boundary-points}] This follows from the standard fact (e.g. \cite{goldman2022geometric} Ex. 4.5.2.3) that a properly convex domain $\Omega$ in $\RP^2$ with a $C^2$ point of nonvanishing curvature contains the ellipse in the $\mathfrak{C}$-closure of $\{[\Omega]\}$. If $\rho \in \text{Hit}_4(S)$, then since $\mathfrak{s}_\rho$ has closed image, this implies that if any leaf $\mathfrak{s}_\rho$ contains a $C^2$ boundary point with nonzero curvature then there is a leaf $\mathfrak{s}_\rho(y)$ projectively equivalent to the ellipse. We conclude $\rho$ is Fuchsian by Theorem \ref{theorem-fuchsian-locus-characterization}. \end{proof}

The proof of Corollary \ref{cor-not-awful} is similar to that of Corollary \ref{cor-arbitrary-boundary-points}.

\appendix
\section{Proofs of Some Useful Facts}

In this appendix, we prove two facts that are used in our main proofs. We reproduce the statements from the body of the paper for the convenience of the reader.

The first is Lemma \ref{regularity-lemma} on models of fixed points of domains invariant under appropriate projective maps.

\begin{lemma}[Models at Fixed Points] Let $\Omega \subset \mathbb{RP}^2$ be a properly convex, strictly convex domain preserved by {\rm{$A \in \text{GL}(3, \mathbb{R})$}} conjugate to {\rm$\text{diag}(\lambda_1, \lambda_2, \lambda_3)$ with $\lambda_1 > \lambda_2 > \lambda_3 > 0$}. Write $l_i = \log \lambda_i $ for $i = 1,2,3$ and let $x_{A^+}$ denote the attracting fixed point of $A$ in $\mathbb{RP}^2$.

Then $x_{A^+} \in \partial \Omega$ is $\alpha$-modelled for $$\alpha = \frac{l_1 - l_3}{l_1 - l_2}.$$
\end{lemma} 

\begin{proof}[Proof of Lemma \ref{regularity-lemma}]
    We follow Benoist's argument in \cite{benoist2004convexesI}, carefully tracking the error terms involved and noting some sources of uniformity to prove the slightly stronger conclusion of $\alpha$-modelling than is asserted in \cite{benoist2004convexesI}. 

    Let $e_1, e_2, e_3$ be the eigenvectors of $A$ with eigenvalues $\lambda_1, \lambda_2, \lambda_3$, respectively. Work in an affine chart for which the repelling hyperplane of $A$ is the hyperplane at infinity, $x_{A^+}$ is at the origin, the attracting hyperplane $y_{A^+}$ of $A^+$ is the horizontal axis, the intersection of the line $\overline{[e_1][e_3]}$ with this affine chart is the vertical axis, and $\Omega$ is contained in the upper half-plane. Strict convexity implies $\partial \Omega$ meets the horizontal axis only at the origin and contains no line segment. Denote by $x_{A^{-}}$ the repelling fixed-point of $A$.
    
    It suffices to produce constants $C_1, C_2$ so that for all $x \neq x_{A^+}$ in a compact subset $K \subset \partial \Omega$ containing $x_{A^+}$ in its interior, \begin{align}\label{regularity-goal} C_1 \leq \log {{d(x, x_{A^+})}} - \alpha^{-1}\log d(x, y_{A^+}) \leq C_2.\end{align}
    
    In this coordinate system, the action of $A$ is given by $\begin{bmatrix} \lambda_2/\lambda_1 & \, \\ \, & \lambda_3/\lambda_1 \end{bmatrix}$. So if $p = (a,b)$, we have $d (A^n p, x_{A^+}) = \frac{1}{\lambda_1^n} (\lambda_2^{2n} |a| + \lambda_3^{2n} |b|)^{1/2}$ and $d(A^np, y_{A^+}) = \lambda_3^{n} b/\lambda_1^n$. For $n$ sufficiently large, we may assume $\lambda_2^{2n}|a|^2 \leq \lambda_2^{2n} |a|^2 + \lambda_3^{2n}|b|^2 \leq \max \{\lambda_{2}^{2n+2}|a|^2, \lambda_2^{2n-2}|a|^2 \}$. So we have that \begin{align*} \log d(A^np, x_{A^+}) &= -{n} \ell_1 + \frac{1}{2} \log ( \lambda_2^{2n} |a|^2 + \lambda_3^{2n}|b|^2) \\ & \leq n(\ell_2 - \ell_1) + |\ell_2 | + \log |a|,
     \end{align*} and the lower bound $n (\ell_2 - \ell_1) + \log |a|$ follows similarly. Furthermore, \begin{align*}
        \alpha^{-1} \log d(A^np, y_{A^+}) &= \alpha^{-1} \log \left( \frac{\lambda_3^n}{\lambda_1^n} |b| \right) = n (\ell_2 - \ell_1) + \alpha^{-1} \log |b|,
     \end{align*} so that for this $p$ we have Equation \ref{regularity-goal}, with $C_1 = \log |a| + \alpha^{-1} \log |b|$ and $C_2 = |\ell_2| + \log|a| + \alpha^{-1} \log |b|$.

    Now, we observe that if $p = (a,b) \in \partial \Omega - \{x_{A^-}\}$ is contained in this affine chart, convexity of $\Omega$ implies that all points in $\partial \Omega - \{ x_{A^{-}} \}$ between $p$ and $gp$ are in the compact box $B = [(\lambda_2/\lambda_1 )a, a] \times [(\lambda_3/\lambda_1 )b, b]$. In particular, the segment of $\partial \Omega$ between $p$ and $x_{A^+}$ is contained in $\{x_{A^+}\} \cup (  \bigcup_{n=0}^\infty A^n B)$. On $B$, all estimates in the above can be taken uniformly and this produces the desired constants. 
\end{proof}

The second and final fact we prove here is the local control on quantitative discreteness of conjugate discrete subgroups of Lie groups that plays a role in our proof of Proposition \ref{key-reduction} (Modify to Continuity).

Recall our notation that $\Lambda$ is a discrete subgroup of a Lie group $G$ equipped with a right-invariant metric, $\kappa(\Lambda) = \text{inf} \{ d(e,g) \mid g \in \Lambda - \{e\} \}$, and conjugation is denoted by $\Psi_g: h \mapsto gh g^{-1}$.

    \begin{lemma}[Discreteness is Conjugation-Stable]\label{appendix-lemma-discreteness-conjugation-stable}
    Let $G$ be a Lie group and $\Lambda < G$ be a discrete subgroup of $G$. Consider the function $\eta: g \mapsto \kappa(\Psi_g(\Lambda))$. Let $g_0 \in G$ be given. Then there is a neighborhood $U$ of $g_0$ so that $\eta(h) > \kappa(\Psi_{g_0}(\Lambda))/3$ for all $h \in U$.
\end{lemma}

\begin{proof}[Proof of Lemma \ref{appendix-lemma-discreteness-conjugation-stable}]
Of course for $h \in G$ we have $\Psi_h(\Lambda) = \Psi_{h g_0^{-1}} (\Psi_{g_0} (\Lambda)).$ By using this, we work with the group $\Lambda' = \Psi_{g_0}(\Lambda)$. It suffices to show:

\textbf{Claim.} There is a neighborhood $U$ of $e$ so that $2\kappa(\Psi_h\Lambda') > \kappa(\Lambda' ) $ for $h\in U$.

For a fixed $R > 0$, let $L_R = \overline{B_R(e)}$ be the closed ball of radius $R$ around the identity, and let $K \subset G$ be compact. Then $(g,p) \mapsto || D_p \Psi_g||$ is continuous. So there is some $C = C(K,R)$ so that $||D_p \Psi_g|| \leq C$ for all $g \in L_R$ and $p \in K$, and in particular $\Psi_{\cdot}(p)$ is $C$-Lipschitz on $L_R$.

Let $r > 0$ and $\epsilon > 0$. Then for all $h\in B_r(e)$ and $g \in B_\epsilon(e)$, we have $\Psi_g(h) \in B_{r + C\epsilon}(e)$. So let $\epsilon < \text{min}(\kappa_{\Lambda'}/3C, R)$ and $r< \kappa_{\Lambda'}/3$. Then for all $g \in B_\epsilon(e)$ and $h \in B_{r}(e)$, we have $\Psi_g(h) \in B_{2\kappa_{\Lambda'}/3}.$ Next, we note that if $\gamma \in G$, $g \in B_\epsilon(e)$, and $\Psi_g(\gamma) \in B_{r}(e)$, then $\gamma = \Psi_{g^{-1}}(\Psi_g(\gamma)) \in B_{2\kappa_{\Lambda'}/3}(e)$, and in particular $\gamma \notin \Lambda'$. We conclude that for all $g$ in $B_\epsilon(e)$ and $\gamma \in \Lambda'$ we have $g^{-1}\gamma g \notin B_{\kappa_{\Lambda}'/3}(e)$, and hence on $U = B_\epsilon(e)$, $\kappa_{g^{-1}\Lambda g} > \kappa_{\Lambda}/3$. \end{proof}

\bibliographystyle{plain}
\bibliography{refs}

\end{document}